\documentclass[letterpaper,10pt]{IEEEtran}

\usepackage{amsmath,graphicx,epsfig,color,amsfonts, algorithm, subfigure, soul,comment}
\usepackage{color}
\usepackage{amssymb}
\usepackage{epsfig}

\usepackage{amsthm}

\graphicspath{{./fig/}}

\usepackage{version,xspace}
\usepackage[noend]{algorithmic}

\def\real{\mathbb{R}}

\newcommand{\until}[1]{\{1,\dots, #1\}}
\newcommand{\subscr}[2]{#1_{\textup{#2}}}
\newcommand{\supscr}[2]{#1^{\textup{#2}}}
\newcommand{\setdef}[2]{\{#1 \; | \; #2\}}

\newcommand{\map}[3]{#1: #2 \to #3}
\newcommand{\union}{\operatorname{\cup}}
\newcommand{\intersection}{\ensuremath{\operatorname{\cap}}}

\newcommand{\subject}{\text{subject to}}

\newcommand{\minimize}{\text{minimize}}
\newcommand{\argmin}{\operatorname{argmin}}

\newcommand\oprocendsymbol{\hbox{$\square$}}
\newcommand\oprocend{\relax\ifmmode\else\unskip\hfill\fi\oprocendsymbol}

\renewcommand{\theenumi}{(\roman{enumi}}

\newcommand\bit[1]{\textit{\textbf{#1}}}

\newcommand\vs[1]{{\color{black}#1}}

\def \bs {\boldsymbol}
\def \mc {\mathcal}

\newtheorem{theorem}{Theorem}
\newtheorem{corollary}[theorem]{Corollary}

\newtheorem{lemma}[theorem]{Lemma}
\newtheorem{assumption}{Assumption}

\newtheorem{proposition}[theorem]{Proposition}

\newtheorem{remark}{Remark}

\title{SIS Epidemic Spreading under Multi-layer Population Dispersal in Patchy Environments
\thanks{This work was supported by ARO grant W911NF-18-1-0325.}
\thanks{An earlier version of this work~\cite{VA-VS:19j} appeared at the 2020 American Control Conference. In addition to the ideas presented in~\cite{VA-VS:19j}, this paper contains alternative proofs, additional analysis, as well as an optimal intervention design for epidemic containment.}
}
\author{Vishal Abhishek and Vaibhav Srivastava
\thanks{V. Abhishek is with the Department of Mechanical Engineering,
       Michigan State University,
       East Lansing, MI 48824-1226, USA
        {\tt\small abhishe3 at egr.msu.edu}}%
\thanks{V. Srivastava is with the Electrical and Computer Engineering, Michigan State University, East Lansing, MI 48824-1226, USA
        {\tt\small vaibhav at egr.msu.edu}}%
}

\begin{document}







\maketitle

\begin{abstract}
We study SIS epidemic spreading models under population dispersal on multi-layer networks. We consider a patchy environment in which each patch comprises individuals belonging to different classes. Individuals disperse to other patches on a multi-layer network in which each layer corresponds to a class. The dispersal on each layer is modeled by a Continuous Time Markov Chain (CTMC).  At each time, individuals disperse according to their CTMC and subsequently interact with the local individuals in the patch according to an SIS model. We establish the existence of various equilibria under different parameter regimes and establish their (almost) global asymptotic stability using Lyapunov techniques. We also derive simple conditions that highlight the influence of the multi-layer network on the stability of these equilibria. For this model, we study optimal intervention strategies using a convex optimization framework. Finally, we numerically illustrate the influence of the multi-layer network structure and the effectiveness of the optimal intervention strategies. 	
\end{abstract}

\begin{IEEEkeywords}
Epidemic models, population dispersal, multi-layer networks, optimal intervention
\end{IEEEkeywords}


\section{Introduction}

Understanding processes underlying the spread of social activities, beliefs, and influence are of foundational interest for the efficient design of socio-technical systems. Deterministic epidemic propagation models have been extensively used to study such spreading processes in the context of epidemics~\cite{anderson1992infectious,DE-JK:10}, the spread of computer viruses~\cite{kleinberg2007computing,wang2009understanding}, routing in mobile communication networks ~\cite{zhang2007performance}, and the spread of rumors~\cite{jin2013epidemiological}. While classical deterministic propagation models~\cite{hethcote2000mathematics} rely on the assumption that the overall population is well-mixed, several generalizations have been explored that incorporate the sparse interaction network among individuals or sub-populations into the model~\cite{lajmanovich1976deterministic,ganesh2005effect,AnalysisandControlofEpidemics_ControlSysMagazine_Pappas,meiBullo2017ReviewPaper_DeterministicEpidemicNetworks,zino2021analysis}. 

In this paper, we focus on a class of such models in which sub-populations are distributed across spatial patches, and individuals interact within the patch using a classical  epidemic propagation model and interact across patches by physically moving to other patches~\cite{wang2004epidemic, jin2005effect,li2009global,arino2005multi}. Individuals may belong to different categories based on their features such as age groups, socioeconomic status, and social preferences, which determine their dispersal pattern. Using Lyapunov techniques, we characterize the steady-state behavior of the model under different parameter regimes and characterize the influence of individuals’ dispersal patterns on the epidemic dynamics.	

Epidemic models have been extensively studied in the literature. The two most widely studied models are SIS (Susceptible-Infected-Susceptible) and SIR (Susceptible-Infected-Recovered) models, wherein individuals are classified into one of the three compartments: susceptible, infected, or recovered, and the dynamics of the fraction of the population in each compartment  is studied~\cite{hethcote2000mathematics}.  Some common generalizations of the SIR/SIS models include SEIR model~\cite{ AnalysisandControlofEpidemics_ControlSysMagazine_Pappas, mesbahi2010graph}, where an additional compartment “exposed” is introduced, SIRS~\cite{DE-JK:10}, where individuals get temporary immunity after recovery and then become susceptible again, and SIRI~\cite{gomez2015abruptTransitionsSIRI, pagliara_NaomiL2018bistability}, where after recovery, agents become susceptible with a different rate of infection.

Networked epidemic models attempt to capture the heterogeneity and structured interactions among subpopulations. Several approaches have been developed to this end. A popular approach to networked epidemic models considers subpopulations as nodes in the interaction graph and models epidemic propagation at each node based on the local subpopulation as well as subpopulations at neighboring nodes~\cite{lajmanovich1976deterministic,ganesh2005effect,AnalysisandControlofEpidemics_ControlSysMagazine_Pappas, meiBullo2017ReviewPaper_DeterministicEpidemicNetworks, fall2007epidemiological, khanafer_Basar2016stabilityEpidemicDirectedGraph,pagliara2020adaptive,pare2020modeling}. These models implicitly assume that the number of individuals at each node remains constant  (if there is no birth or death) and interactions with neighboring nodes do not involve the dispersal of individuals.  Variants of these models with time-varying interaction graphs have also been studied~\cite{bokharaie2010_EpidemicTVnetwork, Preciado2016_EpidemicTVnetwork,  Beck2018_EpidemicTimeVaryingNetwork}.

Another approach to networked epidemic models uses heterogeneous mean-field theory and considers random interaction graphs among individuals. The individuals are classified based on their degree in the interaction graph and the coupled dynamics of individuals in different classes are studied~\cite{vespignani2012modelling, hota2021impacts}. 

The third approach to networked epidemic models considers patchy environments with population dispersal. Each patch contains a subpopulation and the connectivity of the patches is modeled using a dispersal graph. Individuals within each patch interact according to classical epidemic models and disperse to neighboring patches in the dispersal graph~\cite{wang2004epidemic,jin2005effect,li2009global}. Continuum approximations of these models have been studied  using reaction-diffusion processes~\cite{colizza2008epidemicReaction-DiffusionMetapopuln,Saldana2008continoustime_Reaction-DiffnMetapopln}. A closely related class of models is called multi-city models~\cite{lewien2019time,arino2003multi} in which the individuals are identified by their patches and the model keeps track of the distribution of individuals from each patch. Alternative epidemic models with population flows have also been studied~\cite{butler2021effect,possieri2019mathematical,bichara2015sis,ye2020network}.

Multi-layer epidemic models have also been studied. Classical epidemic models have been extended to include heterogeneity due to age groups in epidemic propagation~\cite{hethcote2000mathematics}. Multi-layer networked epidemic models have also been used to study epidemic spread under competitive viruses~\cite{sahneh2014competitive,gracy2022modeling,ye2022convergence}. 

Epidemic spread with dispersal on a multi-layer network of patches has been modeled and studied in~\cite{soriano2018spreading_MultiplexMobilityNetwork_Metapopln}, wherein  at each time, individuals randomly move to another patch, participate in epidemic propagation and then return to their home patch. A multi-species SEIR epidemic model with population dispersal has been studied in~\cite{arino2005multi}, wherein each species may have a different dispersal pattern corresponding to each layer in the multi-layer graph. SIR epidemic dynamics with multi-layer population dispersal have been studied in~\cite{VA-VS:20e}.

The problem of optimal intervention in terms of designing networks and allocating resources for efficiently mitigating epidemics has also received significant attention. A prominent approach leverages a static convex optimization framework to determine optimal intervention parameters such that the convergence rate for the disease-free equilibrium is maximized~\cite{preciado2013convex,preciado2013traffic,drakopoulos2014efficient,wan2007network,nowzari2015optimal}. Some works have also explored dynamic optimization~\cite{khanafer2014information,watkins2019robust} as well as feedback control~\cite{wang2022state} for designing interventions. 

In this paper, we focus on the class of networked epidemic models which consider patchy environments with population dispersal. Individuals may belong to different categories based on their features such as age groups, socioeconomic status, and social preferences, which determine their dispersal pattern. We consider a multi-layer model in which each layer corresponds to a category of individuals and captures the associated dispersal pattern. Individuals within each patch can travel across patches according to a Continuous Time Markov Chain (CTMC) characterizing their dispersal pattern and upon reaching a new patch participate in the local SIS epidemic process. We extend the results for the deterministic network SIS model~\cite{AnalysisandControlofEpidemics_ControlSysMagazine_Pappas, meiBullo2017ReviewPaper_DeterministicEpidemicNetworks, fall2007epidemiological, khanafer_Basar2016stabilityEpidemicDirectedGraph} to the networked SIS model with population dispersal and characterize its steady-state and stability properties. We also study optimal intervention strategies for epidemic containment using this model. To the best of our knowledge, the convex optimization framework for designing optimal intervention has not been applied to epidemic models with population dispersal.

The multi-layer dispersal model in this work is identical to the  population dispersal model used in the multi-species SEIR epidemic model~\cite{arino2005multi}. The model studied in this paper is also similar to the epidemic models under population dispersal studied in~\cite{wang2004epidemic,jin2005effect,li2009global}. In these works, the dispersal patterns depend on the state (susceptible or infected) of the individuals, and conditions for global stability of the disease-free equilibrium and an endemic equilibrium are derived. When the dispersal patterns are identical for all individuals, then these models reduce to a model similar to the single-layer version of the multi-layer model studied in this paper.	In contrast to these works, we focus on the SIS epidemic model with multi-layer population dispersal, and our approach builds upon Lyapunov techniques developed in~\cite{fall2007epidemiological,khanafer_Basar2016stabilityEpidemicDirectedGraph}, which allows us to leverage input-to-state stability ideas and consider non-irreducible interaction graphs.

The major contributions of this paper are fourfold. First, we derive a deterministic continuum limit model describing the interaction of the SIS dynamics with the multi-layer Markovian population dispersal dynamics. The obtained model is similar to the model studied in~\cite{arino2005multi}; however, our presentation derives the model from the first principles. Second, we rigorously characterize the existence and stability of the equilibrium points of the derived model under different parameter regimes. Third, for the stability of the disease-free equilibrium, we determine some useful sufficient conditions which highlight the influence of multi-layer dispersal on the steady-state behavior of the dynamics. Fourth, we study optimal intervention strategies for epidemic containment using our model. Finally, we numerically illustrate that the derived model approximates well the stochastic model with a finite population. We illustrate the influence of the network topology on the transient properties of the model. We also illustrate the effectiveness of optimal intervention strategies.

The remainder of this paper is organized in the following way. In Section \ref{Sec: Mobility Modeling as Continous-Time Markov Process}, we derive the epidemic model under multi-layer dispersal as a continuum limit to two interacting stochastic processes. In Section \ref{sec: analysis}, we characterize the existence and stability of disease-free and endemic equilibrium for the derived model with strongly connected multi-layer interaction. In Section~\ref{sec: analysis-non-strong-connected}, we relax the strong connectivity assumption. We study optimal intervention strategies in Section~\ref{sec:resource-allocation} and draw additional insights on the ideas in the paper through numerical examples in Section \ref{Sec: numerical studies}. Finally, we conclude in Section \ref{Sec: conclusions}.

\subsection*{Mathematical notation}
For any two real vectors $\bs x$, $\bs y \in \real^n$, we denote:\\
$\bs x \gg \bs y$, if $x_i > y_i$ for all $i \in \until{n}$,\\
$\bs x \geq \bs y$, if $x_i \geq y_i$ for all $i \in \until{n}$,\\
$\bs x > \bs y$, if $x_i \geq y_i$ for all $i \in \until{n}$ and $\bs x \neq \bs y$.\\
For a square matrix $G$, spectral abscissa $\map{\mu}{\real^{n\times n}}{\real}$ is defined by 
\[
\mu(G) = \max \setdef{\mathrm{Re}(\lambda)}{\lambda \text{ is an eigenvalue of $G$}}, 
\]
where $\mathrm{Re}(\cdot)$ denotes the real part of the argument. 
Spectral radius $\rho$ is defined by
\[
\rho(G) = \max \setdef{|\lambda|}{\lambda \text{ is an eigenvalue of $G$}}, 
\]
where $|(\cdot)|$ denotes the absolute value of the argument.
For any vector $\bs x = [x_1,\dots,x_n]^\top$, $X=\operatorname{diag}(\bs x)$ is a diagonal matrix with $X_{ii}=x_i$ for all $i \in \until{n}$. Let $\otimes$ denote the Kronecker product and \vs{$\oplus$ denote matrix direct sum.}


\section{SIS Model with Multi-layer Population Dispersal} \label{Sec: Mobility Modeling as Continous-Time Markov Process}

We consider $n$ sub-populations of individuals that are located in distinct spatial regions (patches). We assume the individuals within each patch can be classified into two categories: (i) susceptible, and (ii) infected. We assume that the individuals within each patch are further grouped into $m$ classes, deciding how they travel to other patches. Let the connectivity of these patches corresponding to the dispersal pattern of each class $\alpha \in \until{m}$ be modeled by a digraph $\mc G^\alpha = (\mc V, \mc E^\alpha)$, where $\mc V =\until{n}$ is the node (patch) set and $\mc E^\alpha \subset \mc V \times \mc V$ is the edge set. We model the dispersal of individuals on each graph $\mc G^\alpha$ using a Continuous Time Markov Chain (CTMC) with generator matrix $Q^\alpha$, whose $(i,j)$-th entry is $q^\alpha_{ij}$. The entry $q^\alpha_{i j} \ge 0$, $i \ne j$, is the instantaneous transition rate from node $i$ to node $j$, and $-q^\alpha_{ii}= \nu^\alpha_{i}$ is the total rate of transition out of node $i$, i.e., $\nu^\alpha_{i} = \sum_{j \ne i}q^\alpha_{i j}$. Here, $q^\alpha_{ij} >0$, if $(i,j) \in \mc E^\alpha$; and $q^\alpha_{ij}=0$, otherwise.

We model the interaction of population dispersal with the epidemic process as follows. At each time $t$, individuals of each class $\alpha$ within each node move on graph $\mc G^\alpha$ according to the CTMC with generator matrix $Q^\alpha$ and then interact with individuals within their current node according to an SIS epidemic process.  For the epidemic process at node $i$ \vs{and for individuals in class $\alpha$, let $\beta_i^\alpha >0$ and $\delta_i^\alpha \ge 0$} be the infection and recovery rate, respectively. We let $B^\alpha >0$ and $D^\alpha \ge 0$ be the positive and non-negative diagonal matrices with entries \vs{$\beta_i^\alpha$ and $\delta_i^\alpha$}, $i \in \until{n}$, respectively. \vs{Let $B = \oplus_{\alpha=1}^m B^\alpha$ and $D = \oplus_{\alpha=1}^m D^\alpha$.}


Let $x^\alpha_{i}(t)$ be the number of individuals of class $\alpha$ in patch $i$ at time $t$. Let $p^\alpha_i (t) \in [0,1]$ (respectively, $1-p^\alpha_i (t)$) be the fraction of infected (respectively, susceptible) sub-population of class $\alpha$ at patch $i$. Let 
\begin{align*}
\bs p^\alpha &= \begin{bmatrix}
p^\alpha_{1} \\ \vdots \\p^\alpha_n
\end{bmatrix}, \; 
\bs x^\alpha =\begin{bmatrix}
x^\alpha_{1} \\ \vdots \\ x^\alpha_n
\end{bmatrix}, \; 
\bs p = \begin{bmatrix}
\bs p^1 \\ \vdots \\  \bs p^m
\end{bmatrix},   \text{ and }
 \bs x = \begin{bmatrix}
\bs x^1\\ \vdots\\  \bs x^m
\end{bmatrix}. 
\end{align*}
Define $P^\alpha:=\operatorname{diag}(\bs p^\alpha)$ and $P:=\operatorname{diag}(\bs p)$. For clarity of exposition, we make the following assumption on the connectivity of the dispersal graphs at each layer. We will relax this condition later in Section~\ref{sec: analysis-non-strong-connected}. 

\begin{assumption}[\bit{Strong connectivity}] \label{Assumption:StrongConnectivity}
 Digraph $\mc G^\alpha$ is strongly connected, for all $\alpha \in \until{m}$, which is equivalent to matrices $Q^\alpha$ being irreducible \cite{Bullo-book_Networks}. \vs{A digraph is said to be strongly connected if each node can be reached from every other node by traversing directed edges.} \oprocend
\end{assumption}

Assumption~\ref{Assumption:StrongConnectivity} ensures that the population of every class of individuals in every patch is always non-zero, provided it is initially non-zero. We now derive the continuous-time dynamics that capture the interaction of population dispersal and the SIS epidemic dynamics under Assumption~\ref{Assumption:StrongConnectivity}.

Let \vs{$L = \oplus_{\alpha=1}^m L^\alpha \in \real^{nm \times nm}$}, 
where $L^\alpha(\bs{x}) \in  \real^{n\times n}$ is a matrix with entries 
	\[
	l^\alpha_{ij}(\bs x) = \begin{cases} \sum_{k\neq i}q^\alpha_{k i} \frac{x^\alpha_{k}}{x^\alpha_{i}}, & \text{if } i = j, \\
	-q^\alpha_{j i} \frac{x^\alpha_{j}}{x^\alpha_{i}}, & \text{otherwise},
	\end{cases}
	\]
\vs{	Let $F(\bs x) = \bs 1_m \otimes  \subscr{F}{blk}(\bs x) \in \real^{nm\times nm}$,
where $\subscr{F}{blk}(\bs x) = [F^1 (\bs x),\dots,F^m (\bs x)] \in \real^{n\times nm}$,} and $F^\alpha \in \real^{n \times n}$ is a diagonal matrix with entries $f^\alpha_i (\bs x):=\frac{x^\alpha_i}{\sum_\alpha x^\alpha_i}$, $i \in \until{n}$,  i.e., the fraction of total population at node $i$ contributed by class $\alpha$.




\begin{proposition}[\bit{SIS model with population dispersal}]\label{prop:model}
	Under Assumption~\ref{Assumption:StrongConnectivity}, the dynamics of the fractions of the infected sub-population $\bs p$ and the number of individuals $\bs x^\alpha$ under multi-layer population dispersal model with generator matrices $Q^\alpha$, and infection and recovery matrices $B$ and $D$, respectively, are
	\begin{subequations} \label{eq_Model}
		\begin{align}
		\dot{\bs{p}} & = (BF(\bs{x})-D-L(\bs{x}))\bs{p} - P BF(\bs{x}) \bs{p} \label{eq_p}\\ 
		\dot{\bs{x}}^\alpha & = (Q^\alpha)^\top \bs{x}^\alpha.  \label{eq_x}
		\end{align}
	\end{subequations}
\end{proposition}
\medskip

\begin{proof}
	Consider a small time increment $h>0$ at time $t$. Then the number of individuals of class $\alpha$ present at node $i$ after the evolution of CTMC in time-interval $[t, t+h)$ is
	\begin{equation} \label{eq_popln}
	x^\alpha_{i}(t+h)= x^\alpha_{i}(t)(1-\nu^\alpha_{i}h)+ \displaystyle\sum_{j\neq i}q^\alpha_{j i} x^\alpha_{j}(t)h + o(h) .
	\end{equation}
	After the dispersal, individuals within each node interact according to SIS dynamics. Thus, the \vs{number} of infected \vs{individuals} \vs{of class $\alpha$} present at node $i$ is: 
	\begin{multline} \label{eq_SIS_No of Infected}
\!\!\!\!\!	x^\alpha_{i}(t+h) p^\alpha_{i}(t+h) 
	= - x^\alpha_{i}(t) \delta_{i}^{\vs{\alpha}} p^\alpha_{i}(t)h + x^\alpha_{i}(t)\beta_{i}^{\vs{\alpha}} \supscr {p}{avg}_{i}(t)(1-p^\alpha_{i}(t))h \\
	 + x^\alpha_{i}(t) p^\alpha_{i}(t)(1-\nu^\alpha_{i}h) + \displaystyle\sum_{j\neq i}q^\alpha_{j i} p^\alpha_{j} x^\alpha_{j}(t)h + o(h). 
	\end{multline}
	where $\supscr{p_i}{avg}$ is the fraction of infected population at node $i$ and is given by:
	\[
\supscr {p}{avg}_i = \sum_\alpha f^\alpha_i p^\alpha_i.
	\]
	
	The first two terms on the right side of \eqref{eq_SIS_No of Infected} correspond to the epidemic process within each node, whereas the last two terms correspond to infected individuals coming from other nodes due to dispersal. Using the expression of $x^\alpha_{i}$ from \eqref{eq_popln} in \eqref{eq_SIS_No of Infected} and taking the limit $h \to 0^+$ gives
	\begin{align} \label{eq_SIS_pi}
	\dot{p}^\alpha_{i}= - \delta_{i}^{\vs{\alpha}} p^\alpha_{i} + \beta_{i}^{\vs{\alpha}} \supscr {p}{avg}_{i}(1-p^\alpha_{i})
	-l^\alpha_{ii} p^\alpha_{i} - \displaystyle\sum_{j\neq i} l^\alpha_{i j} p^\alpha_{j} .
	\end{align}
	Writing above in vector form gives:
	\begin{equation}
	\dot{\bs{p}}^\alpha = (-D^\alpha-L^\alpha(\bs{x}^\alpha))\bs{p}^\alpha + B^\alpha  \vs{\subscr{F}{blk}}(\bs{x}) \bs p- P^\alpha B^\alpha  \vs{\subscr{F}{blk}}(\bs{x}) \bs{p} .\label{eq_p_alpha}
	\end{equation}
	Similarly taking limits in \eqref{eq_popln} yields
	\begin{equation} \label{eq_popln_det}
	\dot{x}^\alpha_{i} = -\nu^\alpha_{i} x^\alpha_{i} + \displaystyle\sum_{j\neq i}q^\alpha_{j i} x^\alpha_{j}.
	\end{equation}
	Rewriting \eqref{eq_SIS_pi} and \eqref{eq_popln_det} in vector form establishes the proposition. 
\end{proof}


\vs{
\begin{remark}[\bit{Comparison with the classical networked SIS model}]
We now compare our model~\eqref{eq_Model} with the classical networked SIS dynamics defined on a single-layer network~\cite{AnalysisandControlofEpidemics_ControlSysMagazine_Pappas}:
\begin{equation}\label{eq:network-SIS}
    \dot{\bs p} = (\mathtt{B}-\mathtt{D}) \bs p - P \mathtt{B}\bs p,
\end{equation}
where $\mathtt{B} > 0$ is the infection rate matrix, whose entries $b_{ij}$'s determine the rate at which the infection may be transmitted from node $i$ to node $j$, and $\mathtt{D} > 0$ is a diagonal recovery rate matrix. The model~\eqref{eq:network-SIS} implicitly assumes that the sub-population size at each node is a constant, while in model~\eqref{eq_Model} sub-populations disperse across nodes. For a fair comparison, we specialize model~\eqref{eq_Model} to a single layer and assume that the population dispersal dynamics \eqref{eq_x} operates at its steady state $\bs x^*$. This reduces dynamics~\eqref{eq_p} to 
\begin{equation}\label{eq:single-layer-SIS}
    \dot{\bs p} = (B-D-L(\bs x^*)) \bs p - P B \bs p.
\end{equation}
Since $B$ is a diagonal matrix, it has a smaller number of parameters than the matrix $\mathtt{B}$. Thus, model~\eqref{eq:single-layer-SIS} leverages the population dispersal dynamics, to reduce the number of parameters in the infection rate matrix in a principled way. In several contexts, the data associated with population dispersal may be easier to obtain than the inter-node infection rate data. It is also noteworthy that the non-negative diagonal entries of $L(\bs x^*)$ contribute to increasing the effective recovery rate, while its non-positive off-diagonal entries contribute to the effective inter-node infection rates. 
\end{remark}}

\section{Analysis of SIS Model with Multi-layer Population Dispersal} \label{sec: analysis}

In this section, we analyze the SIS model with multi-layer population dispersal~\eqref{eq_Model} under the following standard assumptions: 

\begin{assumption} \label{Assumption:Dnot0}
\vs{For each layer $\alpha$, there exists a node $k_\alpha$ such that $\delta_{k_\alpha}^\alpha > 0$.} \oprocend
\end{assumption}

Let $\bs \pi^\alpha$ be the right eigenvector of $(Q^\alpha)^\top$ associated with eigenvalue at $0$. We assume that $\bs \pi^\alpha$ is scaled such that its inner product with the associated left eigenvector $\bs 1_{n}$ is unity, i.e., $\bs 1_{n}^\top \bs \pi^\alpha = 1$. Define $Q= \oplus_{\alpha=1}^m Q^\alpha$ and $\bs \pi:=[N^1 (\bs \pi^1)^\top, \dots, N^m (\bs \pi^m)^\top ]^\top $, where $N^\alpha$ is the total number of individuals belonging to class $\alpha$, for $\alpha \in \until{m}$. We call an equilibrium point $(\bs p^*, \bs x^*)$, an endemic equilibrium point if at equilibrium the disease does not die out, i.e., $\bs p^* \neq 0$, otherwise, we call it a disease-free equilibrium point. Let $F^*:=F(\bs x^*)=F(\bs \pi)$ and $L^*:=L(\bs x^*)=L(\bs \pi)$.

\begin{theorem}[\bit{Existence and Stability of Equilibria}] \label{thm:stability}
For the SIS model with multi-layer population dispersal~\eqref{eq_Model} under Assumption~\ref{Assumption:StrongConnectivity}, the following statements hold 
\begin{enumerate}
    \item if $\bs p(0) \in [0,1]^{nm}$, then $\bs p(t) \in [0,1]^{nm}$ for all $t>0$. Also, if $\bs p(0) > \bs 0_{nm}$, then $\bs p(t) \gg \bs 0_{nm}$ for all $t>0$;
    \item the model admits a disease-free equilibrium at $(\bs p^*, \bs x^*)= (\bs 0_{nm}, \bs \pi)$; 
    \item the model admits an endemic equilibrium at $(\bs p^*, \bs x^*) = (\subscr{\bs p}{end}, \bs \pi)$, $\subscr{\bs p}{end} \gg \bs 0$,  if and only if $\mu (BF^*-D-L^*) > 0$; 
    \item the disease-free equilibrium is globally asymptotically stable if and only if $\mu (BF^*-D-L^*) \leq 0$ and is unstable otherwise;
    \item the endemic equilibrium is almost globally asymptotically stable if $\mu (BF^*-D-L^*) > 0$ with region of attraction $\bs p(0) \in [0,1]^{nm} \setminus \{\bs 0_{nm}\}$. 
    \end{enumerate}

\end{theorem}
\medskip

\begin{proof}
The first part of statement (i) follows from the fact that $\dot{\bs{p}}$ is either directed tangent or inside of the region $[0,1]^{nm}$ at its boundary which are surfaces with $p^\alpha_i =0$ or $1$ . This can be seen from \eqref{eq_SIS_pi}. For the second part of (i), rewrite \eqref{eq_p} as:
\begin{equation}\label{eq:thm2-i}
    \dot{\bs{p}} = ((I-P)BF(\bs x)+A(\bs{x}))\bs{p} - E(t) \bs{p}
\end{equation}
where $L(\bs x)=C(\bs x)-A(\bs x)$ with $C(\bs x)$ composed of the diagonal terms of $L(\bs x)$, $A(\bs x)$ is the non-negative matrix corresponding to the off-diagonal terms, and $E(t)=C(\bs x(t))+D$ is a diagonal matrix. Now, consider a variable change $\bs z(t) := e^{\int_{0}^{t}E(\tau) d\tau}\bs p(t)$. Differentiating $\bs z (t)$ and using~\eqref{eq:thm2-i} yields
\begin{align}\label{eq:z}
    \dot{\bs{z}} & 
    =e^{\int_{0}^{t}E(\tau) d\tau}((I-P)BF(\bs x)+A(\bs{x}))e^{\int_{0}^{t}-E(\tau) d\tau} \bs z.
\end{align}
Now, the coefficient matrix of $\bs z$ in~\eqref{eq:z} is always non-negative and strongly connected. The rest of the proof is the same as in \cite[Theorem 4.2 (i)]{meiBullo2017ReviewPaper_DeterministicEpidemicNetworks}.\\
The second statement follows by inspection.\\
The proof of the third statement is presented in Appendix~\ref{Appendix: existence of non-trivial eqb}. 

\noindent\textbf{Stability of disease-free equilibria:}
To prove the fourth statement, we begin by establishing sufficient conditions for instability. The linearization of \eqref{eq_Model} at $(\bs p, \bs x) = (\bs 0, \bs \pi)$ is
\begin{equation} \label{eq_px linear}
    \begin{bmatrix}
     \dot{\bs p} \\
     \dot{\bs x}
     \end{bmatrix} = \begin{bmatrix}
     BF^*-D-L^* & 0 \\
     0 & Q^\top
     \end{bmatrix}\begin{bmatrix}
        \bs p \\
        \bs x
     \end{bmatrix} .
\end{equation}

Since the system matrix in~\eqref{eq_px linear} is block-diagonal, its eigenvalues are the eigenvalues of the block-diagonal sub-matrices. Further, since spectral abscissa $\mu(Q^\top)$ is zero, a sufficient condition for instability of the disease-free equilibrium is that $\mu (BF^*-D-L^*) > 0$.

For the case of $\mu (BF^*-D-L^*) \leq 0$, we now show that the disease-free equilibrium is a globally asymptotically stable equilibrium.
It can be seen from the definitions of matrices $F^*$ and $L^*$, that under Assumption \ref{Assumption:StrongConnectivity}, $(BF^*-D-L^*)$ is an irreducible Metzler matrix. Together with $\mu (BF^*-D-L^*) \leq 0$, implies there exists a positive diagonal matrix $R$ such that 
\[
R(BF^*-D-L^*)+(BF^*-D-L^*)^ \top R = -K,
\]
where $K$ is a positive semi-definite matrix \cite[Proposition 1 (iv), Lemma A.1]{khanafer_Basar2016stabilityEpidemicDirectedGraph}.  Define $\tilde{L} := L(\bs x)-L^*$, $\tilde{F} := F(\bs x)-F^*$ and $r := \|R\|$, where $\|\cdot\|$  denotes the the induced two norm of the matrix. 

Recall $(Q^\alpha)^\top$ has exactly one eigenvalue at zero and all other eigenvalues have negative real parts. The right and left eigenvectors associated with the eigenvalue at zero are $\bs \pi^\alpha$ and $\bs 1_n$, respectively. Let $\bs y^\alpha \in \real^{n-1}$ be the projection of $\bs x^\alpha$ to the orthogonal space of $\bs \pi^\alpha$. 
Let the dynamics~\eqref{eq_x} projected onto the orthogonal space of $\bs \pi^\alpha$ be $\dot{\bs y}^\alpha = \Gamma^\alpha {\bs y^\alpha}$. 
It follows that $\Gamma^\alpha$ is Hurwitz and $\bs y^\alpha = \bs 0_{n-1}$ is a globally asymptotically stable equilibrium point. 
Let $\bs y =[(\bs y^1)^\top,\ldots,(\bs y^m)^\top]^\top$ and $\Gamma = \oplus_\alpha \Gamma^\alpha$. 
Also, with a given population of each class $N^\alpha = \bs 1_n^\top \bs x^\alpha$, we can write $F(\bs x)$ and $L(\bs x)$ as functions of $\bs y$ alone and it can be seen that $\|\tilde{F}(\bs y)\|$ and $\|\tilde{L}(\bs y)\|$ are positive definite functions of $\bs y$. Rewriting dynamics~\eqref{eq_Model} using transformed variables $\bs p$ and $\bs y$, we get
\begin{subequations} \label{cascaded_trivial}
\begin{align}
    \dot{\bs p} &= (BF^*-D-L^*)\bs p-PBF(\bs y)\bs p + (B\tilde{F}(\bs y)-\tilde{L}(\bs y))\bs p \label{p_trivial_y}\\
    \dot{\bs y} &= \Gamma \bs y. \label{cascaded_trivial_y}
\end{align}
\end{subequations}

Consider the Lyapunov function $ V({\bs p}) = {\bs p}^\top  R {\bs p}$. Then
\begin{align}\label{Vdot_cascaded_trivial}
     \dot{V} & = 2 \bs p^\top R \dot{\bs p} \nonumber \\
            & = \bs p^\top (R(BF^*-D-L^*)+(BF^*-D-L^*)^\top R) \bs p \nonumber \\
            & \quad + 2 \bs p^\top R (B\tilde{F}-\tilde{L})\bs p - 2\bs p^\top R P B (\tilde{F}+F^*) \bs p \nonumber \\
            & = -\bs p^\top K \bs p + 2 \bs p^\top R (B(I-P)\tilde{F}-\tilde{L})\bs p - 2\bs p^\top R P B F^* \bs p \nonumber \\
            & \leq - 2 \bs p^\top R P B F^* \bs p + 2k_1(b\|\tilde{F}\|+\|\tilde{L}\|), \nonumber
            \end{align}
where ${k_1, b >0}$ are appropriate positive constants. Pick $\epsilon \in (0,1)$ and it follows that           
            \begin{align}
          \dot V  & \leq - 2 (1-\epsilon)\bs p^\top R P B F^* \bs p - 2 \epsilon \bs p^\top R P B F^* \bs p \nonumber \\
            & \qquad \quad + 2k_1(b\|\tilde{F}\|+\|\tilde{L}\|) \nonumber \\
            & \leq - 2 (1-\epsilon)\bs p^\top R P B F^* \bs p - k_2\|\bs p\|^3_\infty + 2k_1(b\|\tilde{F}\|+\|\tilde{L}\|) \nonumber\\
            & \leq - 2 (1-\epsilon)\bs p^\top R P B F^* \bs p - k_2\|\bs p\|^3_\infty + \rho_1(\|\bs y\|_\infty) \nonumber\\
            & \leq - 2 (1-\epsilon)\bs p^\top R P B F^* \bs p ,
\end{align}
for $\|\bs p\|_\infty \geq \rho_2(\|\bs y\|_\infty)=: \sqrt[3]{\frac{\rho_1(\|\bs y\|_\infty)}{k_2}}$.
Here, $k_2>0$ is an appropriate constant and $\rho_1$ is a suitable class $\mc K$ function. Since $ 2k_1(b \|\tilde{F}\|+\|\tilde{L}\|)$ is a positive definite function of $\bs y$, existence of $\rho_1$ follows from~\cite[Lemma 4.3]{HKK:02}. $\rho_2$ 
is also a class $\mc K$ function. 
Since \eqref{Vdot_cascaded_trivial} is a continuous negative definite function of $\bs p \in [0,1]^{nm}$, it follows from~\cite[Theorem 4.19]{HKK:02} that dynamics \eqref{p_trivial_y} is ISS (Input to state stable) with respect to input $\bs y$. Note that the origin is globally asymptotically stable for \eqref{cascaded_trivial_y}. Now using~\cite[Lemma 4.7]{HKK:02} for the cascaded system \eqref{p_trivial_y}-\eqref{cascaded_trivial_y}
implies that the origin is globally asymptotically stable equilibrium for the cascaded system.

\noindent\textbf{Stability of endemic equilibria:}
Finally, we prove the fifth statement. To this end, we first establish an intermediate result. 
\begin{lemma} \label{Lemma:p_i tends to 0 implies p tends to 0}
For the dynamics~\eqref{eq_p}, if $p^\alpha_{i}(t) \to 0$ as $t \to \infty$, for some $i \in \until{n}$ and $\alpha \in \until{m}$, then $\bs p(t) \to \bs 0$ as $t \to \infty$.
\end{lemma} 
\begin{proof}
It can be easily seen from \eqref{eq_SIS_pi} that $\ddot{p}^\alpha_{i}$ is bounded and hence $\dot{p}^\alpha_{i}$ is uniformly continuous in $t$. Now if $p^\alpha_{i}(t) \to 0$ as $t \to \infty$, it follows from Barbalat's lemma \cite[Lemma 4.2]{slotine1991applied} that $\dot{p}^\alpha_{i} \to 0$. Therefore, from \eqref{eq_SIS_pi} and the fact that $- l^\alpha_{i j}(\bs x) \geq 0$ and $p^\alpha_{i} \geq 0$, it follows that $p^\alpha_{j}(t) \to 0$ for all $j$ such that $- l^\alpha_{i j} (\bs x) \neq 0$. Using Assumption~\ref{Assumption:StrongConnectivity} and applying the above argument for each class at each node implies
$\bs p(t) \to \bs 0$. 
\end{proof}
Define $\tilde{\bs p} := \bs p-\bs p^*$, $P^* := \operatorname{diag}(\bs p^*)$ and $\tilde{P} := \operatorname{diag}(\tilde{\bs p})$. Then
\begin{equation*}
    \begin{split}
        \dot{\tilde{\bs p}} & =  (BF-D-L- P BF) \bs p \\
                        & =  (BF^*-D-L^*- P^* BF^*) \bs p^*  \\  &  \quad +  (BF^*-D-L^*- P^* BF^*) \tilde{\bs p} \\
                        & \quad + (B\tilde{F}- \tilde{L}) \bs p - PB\tilde{F} \bs p - \tilde{P}BF^* \bs p \\
                        & = ((I- P^*)BF^*-D-L^*) \tilde{\bs p} + ((I-P)B\tilde{F} - \tilde{L}) \bs p \\ 
                        &\quad   - \tilde{P}BF^* \bs p,
    \end{split}
\end{equation*}
where $(BF^*-D-L^*- P^* BF^*) \bs p^* = \bs 0$, as ($\bs p^*$, $\bs x^*$) is an equilibrium point.

Note that $(BF^*-D-L^*- P^* BF^*)=((I-P^*)BF^*-D-L^*)$ is an irreducible Metzler matrix and $\bs p^* \gg \bs 0$ is an eigenvector associated with eigenvalue at zero. Therefore, the Perron-Frobenius theorem for irreducible Metzler matrices \cite{Bullo-book_Networks} implies 
$\mu ((I- P^*)BF^*-D-L^*) = 0$. Also, this means there exists a positive-diagonal matrix $R_2$ and a positive semi-definite matrix $K_2$ such that
\begin{multline*}
   R_{2}((I-P^*)BF^*-D-L^*) \\
   +((I-P^*)BF^*-D-L^*)^\top R_{2} = -K_2 .
\end{multline*}
Consider the cascaded system
\begin{subequations} \label{cascaded_nontrivial}
\begin{align}
     \dot{\tilde{\bs p}} &= ((I- P^*)BF^*-D-L^*) \tilde{\bs p} + ((I-P)B\tilde{F}(\bs y) \nonumber \\ 
       & \qquad \quad  - \tilde{L}(\bs y)) \bs p 
  - \tilde{P}BF^* \bs p \\
    \dot{\bs y} &= \Gamma \bs{y}. \label{cascaded_nontrivial_y}
\end{align}
\end{subequations}

Similar to the proof of the fourth statement, take $V_{2}(\tilde{\bs p}) = \tilde{\bs p}^\top R_{2} \tilde{\bs p}$.
Then,
\begin{align*}
           \dot{V_{2}} & = 2 \tilde{\bs p}^\top R_{2} \dot{\tilde{\bs p}}  \\
            & = \tilde{\bs p}^\top (R_{2}((I-P^*)BF^*-D-L^*)\\
            &\quad +((I-P^*)BF^*-D-L^*)^\top R_{2}) \tilde{\bs p} \\
            & \quad + 2 \tilde{\bs p}^\top R_{2}((I-P)B\tilde{F}-\tilde{L})\bs p - 2 \tilde{\bs p}^\top R_{2} \tilde{P} B F^*\bs p \\
            & = -\tilde{\bs p}^\top K_{2} \tilde{\bs p} + 2 \tilde{\bs p}^\top R_{2}((I-P)B\tilde{F}-\tilde{L})\bs p \\
            & \quad - 2 \tilde{\bs p}^\top R_{2} \tilde{P} B F^*\bs p \\
            & \leq - 2 \tilde{\bs p}^\top R_{2} \tilde{P} BF^* \bs p + k_2 (b\|\tilde{F}\| + \|\tilde{L}\|),\; {k_2, b >0}
\end{align*}   
where ${k_2, b >0}$ are appropriate positive constants. Pick $\epsilon \in (0,1)$ and it follows that  
\begin{align}  
    \dot{V_{2}}           & \leq - 2(1-\epsilon)\tilde{\bs p}^\top R_{2} \tilde{P} BF^* \bs p + 2\epsilon\tilde{\bs p}^\top R_{2} \tilde{P} BF^* \bs p +\rho_4(\|\bs y\|) \nonumber\\
            & \leq - 2(1-\epsilon)\tilde{\bs p}^\top R_{2} \tilde{P} BF^* \bs p - \rho_3(\| \tilde{\bs p}\|) + \rho_4(\|\bs y\|) \nonumber \\
            & \leq - 2(1-\epsilon)\tilde{\bs p}^\top R_{2} \tilde{P} BF^* \bs p, \label{Vdot_cascaded_non-trivial}   
\end{align}
for ${\| \tilde{\bs p}\|\geq \rho_5(\|\bs y\|)}=: \rho_3^{-1}\rho_4(\|\bs y\|)$.
Here, $\rho_4$ and $\rho_5$ are suitable class $\mc K$ functions and $\rho_3$ a class $\mc K_\infty$ function. The existence of these functions follows from~\cite[Lemma 4.3]{HKK:02}. The remainder of the proof follows analogously to the proof of the fourth statement. 
\end{proof}

\begin{remark}
It should be noted that dynamics~\eqref{eq_Model} has an equilibrium point on the boundary of its domain  and the standard Lyapunov analysis assumes that the equilibrium point lies in the interior of the domain. However, since we have shown that $[0,1]^{nm}$ is the invariant set for the $\bs p$ dynamics, similar conclusions can be drawn using LaSalle's invariance principle~\cite{HKK:02}. Similarly, the input-to-state stability~\cite[Lemma 4.7)]{HKK:02} cannot be directly used. However, its proof relies on \cite[Theorem 4.18 and 4.19]{HKK:02}, which can be extended to our positive system by constructing invariant sets that are the intersection of $[0,1]^{nm}$ and the invariant sets in the proof of~\cite[Theorem 4.18 and 4.19]{HKK:02}. 
\end{remark}

\begin{corollary}[\bit{Stability of disease-free equilibria}] \label{cor:dis-free}
For the SIS epidemic model with multi-layer population dispersal~\eqref{eq_Model} under Assumption~\ref{Assumption:StrongConnectivity} and the disease-free equilibrium $(\bs p^*, \bs x^*)= (\bs 0, \bs \pi)$ the following statements hold
\begin{enumerate}
    \item a necessary condition for stability is that for each $i \in \until{n}$, $\exists \alpha \in \until{m}$ such that \vs{$\delta_{i}^\alpha > \beta_{i}^\alpha - \nu^\alpha_{i}$}; 
    \item  a necessary condition for stability is that there exists some $i \in \until{n}$ \vs{and $\alpha \in \until{m}$ such that $\delta_i^\alpha \geq \beta_i^\alpha$}; 
    \item a sufficient condition for stability is \vs{$\delta_{i}^\alpha \geq \beta_{i}^\alpha$, for each $i \in \until{n}$ and $\alpha \in \until{m}$}; 
    \item a sufficient condition for stability is 
    \[
    \frac{\lambda_{2}}{\Big(1+\sqrt{1+\frac{\lambda_{2}}{\sum_{i} w_{i}\big(\delta_{i}-\beta_{i}-s\big)}}\Big)^2 nm + 1} + s \geq 0,
    \]
    where 
    $\bs w$ is a positive left eigenvector of $(BM+L^*)$ such that $\bs w^\top(BM+ L^*) = 0$ with $\max_{i} w_{i} = 1$, $M= I-F^*$, \vs{$s = \min_{i, \alpha} (\delta_{i}^\alpha-\beta_{i}^\alpha)$}, $W = \operatorname{diag} (\bs w)$, and  $\lambda_{2}$ is the second smallest eigenvalue of $\frac{1}{2}\big(W(BM+ L^*) +(BM+ L^*)^\top W\big)$.
\end{enumerate}
\end{corollary}
\begin{proof}
We begin by proving the first two statements. First, we note that $(L^\alpha)^*_{ii} = \nu^\alpha_i$. This can be verified by evaluating $L^*=L(\bs \pi)$ and utilizing the fact that $Q^\top \bs \pi = \bs 0$. The necessary and sufficient condition for the stability of disease-free equilibrium is $\mu (BF^*-D-L^*) \leq 0$. Note that $BF^*-D-L^*$ is an irreducible Metzler matrix. Perron-Frobenius theorem for irreducible Metzler matrices implies that there exists a real eigenvalue equal to $\mu$ with a positive eigenvector, i.e.,
$(BF^*-D-L^*)\bs z = \mu \bs z $, where $\bs z \gg \bs 0 $. Rename components of $\bs z$ as $z_{(n(\alpha-1)+i)}=z^\alpha_i$, $\alpha \in \until{m}$, $i\in \until{n}$, to write $\bs z = [(\bs z^1)^\top,\dots,(\bs z^m)^\top]^\top$. Let for each $i \in \until{n}$, $z^{\kappa_i}_i=\min \{z^1_i,\dots,z^m_i\}$. Since $\mu \leq 0$, written component-wise for $(n(\kappa_i-1)+i)$-th component of $(BF^*-D-L^*)\bs z$:
\vs{\begin{align*}
        & \sum_\alpha \beta_{i}^{\kappa_i}f^{*\alpha}_i z^\alpha_i - (\delta_{i}^{\kappa_i} +\nu^{\kappa_i}_{i})z^{\kappa_i}_{i} - \displaystyle\sum_{j\neq i} l^{*\kappa_i}_{i j}z^{\kappa_i}_j \leq 0 \nonumber  \\
        & \implies (\beta_{i}^{\kappa_i} - \delta_{i}^{\kappa_i}-\nu^{\kappa_i}_i)z^{\kappa_i}_{i} \leq -\sum_\alpha \beta_{i}^{\kappa_i}f^{*\alpha}_i (z^\alpha_i-z^{\kappa_i}_i) \nonumber\\
      & \qquad \qquad \qquad   \qquad \qquad \qquad  \qquad \qquad + \displaystyle\sum_{j\neq i} l^{*\kappa_i}_{i j}z^{\kappa_i}_j  \nonumber \\
        & \implies \beta_{i}^{\kappa_i} - \delta_{i}^{\kappa_i}-\nu^{\kappa_i}_i < 0 .
    \end{align*}}
Here, we have used facts: $\sum_\alpha f^{*\alpha}_i = 1$, $f^{*\alpha}_i>0 $, $ l^{*\kappa_i}_{ij}\leq 0$ and that there exists $j \in \until{n}$ such that $ l^{*\kappa_i}_{ij}<0$. This proves the first statement. 

The second statement follows similarly to the first statement by selecting $z_i^{\kappa_i} =\min\{z^1_1,\dots,z^m_n\}$ and using the additional fact $\nu^{\kappa_i}_i+ \sum_{j\neq i} l^{*\kappa_i}_{i j}=0$. 



Let $F^*=I-M$ where $M$ is the Laplacian matrix which can be seen from the definition of $F$. Now $BF^*-D-L^* = B-D-(BM+L^*)$. Since $(BM+L^*)$ is a Laplacian matrix, if \vs{$\delta_i^\alpha \geq \beta_i^\alpha$, for each $i \in \until{n}$ and $\alpha \in \until{m}$}, from Gershgorin disks theorem \cite{Bullo-book_Networks}, $\mu \leq 0$, which proves the third statement.

For the last statement, we use an eigenvalue bound for perturbed irreducible Laplacian matrix of a digraph~ \cite[Theorem 6]{wu2005bounds}, stated below:

Let $H = A + \Delta$, where $A$ is an $n\times n$ irreducible Laplacian matrix and $\Delta \neq 0$ is a non-negative diagonal matrix. Then  
\begin{equation*}
\begin{split}
  \mathrm{Re}(\lambda(H)) \geq \frac{\lambda_{2}}{\Big(1+\sqrt{1+\frac{\lambda_{2}}{\sum_{i} w_{i}\Delta_i}}\Big)^2 n + 1} > 0,
  \end{split}
\end{equation*}
where, $\bs w$ is a positive left eigenvector of $A$ such that $\bs w^\top A = 0$ with $\max_{i} w_{i} = 1$, $W = \operatorname{diag} (\bs w)$, and  $\lambda_{2}$ is the second smallest eigenvalue of $\frac{1}{2}(W A + A^\top W)$.\\
Now, in our case, necessary and sufficient condition for the stability of disease-free equilibrium is:
\begin{equation*}
\begin{split}
  \mathrm{Re}(\lambda(BM+L^*+D-B)) & = \mathrm{Re}(\lambda(BM+L^*+\Delta + sI)) \\
  & = \mathrm{Re}(\lambda(BM+L^*+\Delta)) + s \\
  & \geq 0,
  \end{split}
\end{equation*}
where, \vs{$s = \min_{i, \alpha} (\delta_{i}^\alpha -\beta_{i}^\alpha)$} and $\Delta=D-B-sI$. Applying the eigenvalue bound with $H=BM+L^*+\Delta$ gives sufficient condition (iv). 
\end{proof}

\medskip
\begin{remark}
For a given multi-layer graph and the associated dispersal transition rates in dynamics~\eqref{eq_Model}, let \vs{$s = \operatorname{min}_{i, \alpha} (\delta_{i}^\alpha-\beta_{i}^\alpha)$ and $(i^*, \alpha^*)= \argmin_{i, \alpha} (\delta_{i}^\alpha-\beta_{i}^\alpha)$. Then, there exist $\delta_i^\alpha$'s, $(i, \alpha)\neq (i^*,\alpha^*)$}, that satisfy statement (iv) of Corollary \ref{cor:dis-free} if $s > \subscr{s}{lower}$, where 
\[
\subscr{s}{lower}=-\frac{\lambda_2}{4mn+1}.
\] \oprocend
\end{remark}

\begin{remark}(\bit{Influence of population dispersal on the stability of disease-free equilibrium.})
The statement (iv) of Corollary~\ref{cor:dis-free} characterizes the influence of population dispersal on the stability of disease-free equilibria. In particular, $\lambda_2$ is a measure of the ``intensity" of population dispersal and $s$ is a measure of the largest deficit in the recovery rate compared with the infection rate among nodes. The sufficient condition in statement (iv) states explicitly how population dispersal can allow for stability of disease-free equilibrium even under a deficit in recovery rate at some nodes. \oprocend
\end{remark}

\section{SIS model under multi-layer population dispersal: Non-strongly connected layers} \label{sec: analysis-non-strong-connected}
In this section, we relax Assumption  \ref{Assumption:StrongConnectivity}, so that the digraph representing population dispersal on a layer need not be strongly connected. We first derive a reduced model and then present a stability analysis.  

\subsection{Reduced model for non-strongly connected layers}
As discussed in~\cite[Chapter 10]{Bullo-book_Networks}, a CTMC is a closed dynamical flow system and accordingly every individual of each class $\alpha$ eventually resides within the sinks of $\mc G^\alpha$~\cite[Theorem 10.13]{Bullo-book_Networks}. Let $\mc V^\alpha = \{\bar v_1^\alpha, \ldots, \bar v_{n^\alpha}^\alpha\} \subset \until{n}$ be the nodes in the sinks of $\mc G^\alpha$ and let 
$\bar{\mc G}^\alpha$ be the subgraph of $\mc G^\alpha$ induced by these nodes. Let $\mc V \setminus \mc V^\alpha = \{\hat v_1^\alpha, \ldots, \hat v_{(n-n^\alpha)}^\alpha\}$.  


 Let $\bar p_i^\alpha(t) = p^\alpha_{\bar v_i^\alpha}(t)$, $\bar x_i^\alpha(t) = x^\alpha_{\bar v_i^\alpha}(t)$, $\hat p_i^\alpha(t) = p^\alpha_{\hat v_i^\alpha}(t)$, and $\hat x_i^\alpha(t) = x^\alpha_{\hat v_i^\alpha}(t)$. Define 
\begin{align*}
\bar{\bs p}^\alpha &= \begin{bmatrix}
\bar p^\alpha_{1} \\ \vdots \\ \bar p^\alpha_{n^\alpha}
\end{bmatrix}, \; 
\bar{\bs x}^\alpha =\begin{bmatrix}
\bar x^\alpha_{1} \\ \vdots \\ \bar x^\alpha_{n^\alpha}
\end{bmatrix}, \;
\bar{\bs p} = \begin{bmatrix}
\bar{\bs p}^1 \\ \vdots \\  \bar{\bs p}^m
\end{bmatrix}, \; 
 \bar{\bs x} = \begin{bmatrix}
\bar{\bs x}^1\\ \vdots\\  \bar{\bs x}^m
\end{bmatrix}, \\
\hat{\bs p}^\alpha &= \begin{bmatrix}
\hat p^\alpha_{1} \\ \vdots \\ \hat p^\alpha_{n- n^\alpha}
\end{bmatrix}, \; 
\hat{\bs x}^\alpha =\begin{bmatrix}
\hat x^\alpha_{1} \\ \vdots \\ \hat x^\alpha_{n-n^\alpha}
\end{bmatrix}, \; 
\hat{\bs p} = \begin{bmatrix}
\hat{\bs p}^1 \\ \vdots \\  \hat{\bs p}^m
\end{bmatrix}, 
\hat{\bs x} = \begin{bmatrix}
\hat{\bs x}^1\\ \vdots\\  \hat{\bs x}^m
\end{bmatrix}. 
\end{align*}

Let $\bar L^\alpha (\bs x) \in \real^{n^\alpha \times n^\alpha}$ be the matrix with entries
	\[
	(\bar l^\alpha)_{ij}(\bs x) = \begin{cases} \sum_{h \in \mc V, h\neq \bar v_i^\alpha}q^\alpha_{h {\bar v_i^\alpha}} \frac{x^\alpha_{h}}{\bar x^\alpha_{i}}, & \text{if } \bar v_i^\alpha = \bar v_j^\alpha, \\
	-q^\alpha_{{\bar v_j^\alpha} {\bar v_i^\alpha}} \frac{\bar x^\alpha_{j}}{\bar x^\alpha_{i}}, & \text{otherwise},
	\end{cases}
	\]
for each $\bar v_i, \bar v_j \in \mc V^\alpha$. Likewise, let ${\hat L}^\alpha (\bs x) \in \real^{n^\alpha \times (n- n^\alpha)}$ be the matrix with entries 
	\[
	(\hat l^\alpha)_{ij}(\bs x) = 
	-q^\alpha_{\hat v_j^\alpha \bar v_i^\alpha} \frac{\hat x^\alpha_{j}}{\bar x^\alpha_{i}}, \text{ for } \bar v_i^\alpha \in \mc V^\alpha \text{ and  } \hat v_j^\alpha \in \mc V \setminus \mc V^\alpha. 
	\]
Let $\bar L = \oplus_{\alpha =1}^m \bar L^\alpha$ and $\hat L = \oplus_{\alpha =1}^m \hat L^\alpha$, where $\oplus$ denotes matrix direct sum. Let $\bar n = \sum_{\alpha =1}^m n^\alpha$. 
Let $\bar F$ be a block matrix with blocks $\bar F^{\alpha \sigma} \in \real^{n_\alpha \times n_\sigma}$, $\alpha, \sigma \in \until{m}$. Define the entries of $\bar F^{\alpha \sigma}$ by
\[
\bar f_{ij}^{\alpha \sigma} = \begin{cases}
\frac{\bar x^\sigma_j}{\sum_{\alpha=1}^m x_{\bar v_j^\sigma}^\alpha}, & \text{ if } \bar v_i^\alpha = \bar v_j^\sigma , \\
0, & \text{ otherwise.}
\end{cases}
\]
Likewise, let $\hat F$ be a block matrix with blocks $\hat F^{\alpha \sigma} \in \real^{n_\alpha \times (n-n_\sigma)}$, $\alpha, \sigma \in \until{m}$. Define the entries of $\hat F^{\alpha \sigma}$ by
\[
\hat f_{ij}^{\alpha \sigma} = \begin{cases}
\frac{\hat x^\sigma_j}{\sum_{\alpha=1}^m x_{\hat v_j^\sigma}^\alpha}, & \text{ if } \hat v_j^\sigma = \bar v_i^\alpha, \\
0, & \text{ otherwise.}
\end{cases}
\]
Let \vs{$\bar D^\alpha = \operatorname{diag}(\delta_{\bar v_1^\alpha}^\alpha, \ldots,  \delta_{\bar v_{n^\alpha}^\alpha}^\alpha )$, $\bar B^\alpha = \operatorname{diag}(\beta_{\bar v_1^\alpha}^\alpha, \ldots,  \beta_{\bar v_{n^\alpha}^\alpha}^\alpha )$}, $\bar D = \oplus_{\alpha=1}^m \bar D^\alpha$ and $\bar B = \oplus_{\alpha=1}^m \bar B^\alpha$. Let $\bar P = \operatorname{diag}(\bar{\bs p})$ and $\hat P = \operatorname{diag}(\hat{\bs p})$.


\begin{corollary}[\bit{SIS model under multilayer population dispersal: Non-strongly connected layers}]\label{corr:reduced-model}
Consider the multi-layer population dispersal model defined by a CTMC for each layer $\alpha \in \until{m}$ with generator matrices $Q^\alpha$. Let $\mc V^\alpha$ be the set of sink nodes for the CTMC associated with layer $\alpha$. Then,  the dynamics of the fractions of the infected sub-population restricted to sink nodes at each layer $\bar{\bs p}$ and the number of individuals $\bs x^\alpha$ are
	\begin{subequations} \label{eq_Model_red}
		\begin{align}
		\dot{\bar{\bs{p}}} & = (\bar B \bar F(\bs{x})-\bar D- \bar L(\bs{x}))\bar{\bs{p}} - \bar P \bar B \bar F(\bs{x}) \bar{\bs{p}} \nonumber \\
		& \qquad + (I -\bar P) \bar B  \hat F(\bs x) \hat{\bs p} - \hat L (\bar{\bs x}) \hat{\bs p} 
		\label{eq_p_bar}\\ 
		\dot{\bs{x}}^\alpha & = (Q^\alpha)^\top \bs{x}^\alpha. \label{eq_x_bar}
		\end{align}
	\end{subequations}
\end{corollary}
\begin{proof}
The proof follows by restricting~\eqref{eq_Model} to sinks in each layer and carefully accounting for the contribution of nodes not in the sink component. 
\end{proof}

\vs{An illustration of the reduction in Corollary~\ref{corr:reduced-model} for a small bi-layer graph is presented in Appendix~\ref{app:low-dim}.}

\subsection{Stability Analysis}
Note that at equilibrium every $x_k^\alpha$ for $k \in \mc V \setminus \mc V^\alpha$ is zero and hence, the block diagonal matrix $\bar L(\bs x)$, at equilibrium, comprises strongly connected Laplacian matrices as the diagonal blocks. Let $\bar L^*$ and $\bar F^*$ denote the equilibrium values of $\bar L(\bs x)$ and $\bar F(\bs x)$.
Then, similar to Theorem~\ref{thm:stability}, 
the matrix $(\bar{B}\bar F^*-\bar{D}-\bar L^*)$ determines the existence and nature of equilibrium point. The matrix $(\bar{B}\bar{F}^*-\bar{D}-\bar L^*)$ may comprise multiple diagonal blocks that are irreducible. Consider a scenario with a two-layered graph in which layer 1 has two strongly connected sinks with node sets $\mc V^1_1$ and $\mc V^1_2$, and layer 2 also has two strongly connected sinks with node sets $\mc V^2_1$ and $\mc V^2_2$. Suppose $\mc V^1_2 \intersection \mc V^2_2$ is an empty set, while $\mc V^1_1 \intersection \mc V^2_1$ is non-empty. Then, $(\bar{B}\bar{F}^*-\bar{D}-\bar L^*)$ will have three irreducible blocks associated with $\mc V^1_1 \union \mc V^2_1$, $\mc V^1_2$ and $\mc V^2_2$. Since each of these blocks can be individually analyzed, for simplicity, we assume that   $(\bar{B}\bar{F}^*-\bar{D}-\bar L^*)$ is irreducible. 

\begin{assumption}\label{assump:sink-recovery}
The matrix $(\bar{B}\bar{F}^*-\bar{D}-\bar L^*)$ is irreducible. Additionally, 
for every strongly connected sink component in each layer $\alpha$, there exists a node $k$ such that \vs{$\delta_k^\alpha >0$}. \oprocend 
\end{assumption}



\begin{corollary}[\bit{Stability of equilibria with non-strongly connected layers}] \label{thm:stability-non-strong}
For the reduced SIS model with population dispersal on non-strongly connected multiple layers~\eqref{eq_Model_red}, under Assumption~\ref{assump:sink-recovery}, the following statements hold 
\begin{enumerate}
    \item the disease-free equilibrium $(\bar{\bs p}, \bs x) = (\bs 0_{\bar n}, \bs \pi)$ is globally asymptotically stable if and only if $\mu (\bar B \bar F^*- \bar D- \bar L^*) \leq 0$ and is unstable otherwise;
    \item the endemic equilibrium is almost globally asymptotically stable if $\mu (\bar B \bar F^*- \bar D- \bar L^*) > 0$ with region of attraction $\bar{\bs p}(0) \in [0,1]^{\bar n}$ such that $\bs p(0) \neq \bs 0_{\bar n}$.
    \end{enumerate}

\end{corollary}

\begin{proof}
We start with establishing the first statement. 
Let $\bs c_1 = \hat F(\bs x) \hat{\bs p}$ and $\bs c_2 = \hat L (\bar{\bs x}) \hat{\bs p}$. Note that entries of $\hat F$ and $\hat L$ are either zero or exponentially converge to zero, and $\hat{\bs p} \in [0,1]$. Therefore, $\bs c_1$ and $\bs c_2$ exponentially converge to the origin.

Under Assumption~\ref{assump:sink-recovery}, $(\bar{B}\bar{F}^*-\bar{D}-\bar L^*)$ is an irreducible Metzler matrix. Together with $\mu (\bar{B}\bar{F}^*-\bar{D}-\bar L^*) \leq 0$, implies there exists a positive diagonal matrix $\bar R$ such that 
\[
\bar R(\bar{B}\bar{F}^*-\bar{D}-\bar L^*)+(\bar{B}\bar F^*-\bar{D}-\bar L^*)^\top \bar R) =-\bar K,
\]
where $\bar K$ is a positive semi-definite matrix. 

Consider the Lyapunov function $\bar V(\bar{\bs p}) = \bar{\bs p}^\top \bar R \bar{\bs p}$. Then
\begin{align}\label{Vdot_cascaded_trivial_unconnected}
     \dot{\bar V} & = 2 \bar{\bs p}^\top R \dot{\bar{\bs p}} \nonumber \\
            & = \bar{\bs p}^\top (\bar R(\bar{B}\bar{F}^*-\bar{D}-\bar L^*)+(\bar{B}\bar F^*-\bar{D}-\bar L^*)^\top \bar R) \bar{\bs p} \nonumber \\
            & \quad + 2 \bar{\bs p}^\top \bar R (\bar{B}\tilde{\bar F} -\tilde{\bar L}) \bar{\bs p} - 2 \bar{\bs p}^\top \bar R \bar{P} \bar{B} (\tilde{\bar F}+ \bar F^*) \bar{\bs p} \nonumber \\ 
            & \quad + 2 \bar{\bs p}^\top (I-\bar{P})\bar{B}\bs c_1+ 2 \bar{\bs p}^\top \bs c_2 \nonumber \\
            & = - \bar{\bs p}^\top \bar K \bar{\bs p} + 2 \bar{\bs p}^\top \bar R (\bar B(I-\bar P)\tilde{\bar F}-\tilde{\bar L})\bar{\bs p} - 2 \bar{\bs p}^\top \bar R \bar P \bar{B} \bar F^* \bar{\bs p} \nonumber \\
            & \quad + 2 \bar{\bs p}^\top (I-\bar P)\bar{B}\bs c_1+ 2 \bar{\bs p}^\top \bs c_2 \nonumber \\
            & \leq - 2 \bar{\bs p}^\top \bar R \bar P\bar{B} \bar F^* \bar{\bs p} + 2\bar k_1(b_1\|\tilde{\bar F}\|+b_2\|\tilde{\bar L}\|+b_3\|\bs c_1\|+\|\bs c_2\|), \nonumber 
            \end{align}
where $\tilde{\bar F}(\bs x) = \bar F(\bs x) -\bar F^* $, $\tilde{\bar L}(\bs x) = \bar L(\bs x) - \bar L^* $, and ${\bar k_1, b_1, b_2, b_3 >0}$ are appropriate positive constants. Pick $\epsilon \in (0,1)$. Then, it follows that 
            \begin{align}
              \dot{\bar V}  & \leq - 2 (1-\epsilon)\bar{\bs p}^\top \bar R \bar P \bar{B} \bar F^* \bar{\bs p} - 2 \epsilon \bar{\bs p}^\top \bar R \bar P \bar{B} \bar F^* \bar{\bs p} \nonumber \\
              &\quad + 2 \bar k_1(b_1\|\tilde{\bar F}\|+b_2\|\tilde{\bar L}\|+b_3\|\bs c_1\|+\|\bs c_2\|) \nonumber \\
            & \leq - 2 (1-\epsilon)\bar{\bs p}^\top \bar R \bar P \bar{B} \bar F^* \bar{\bs p} 
            - \bar k_2\|\bar{\bs p}\|^3_\infty \nonumber \\
            &\quad + 2\bar k_1(b_1\|\tilde{\bar F}\|+b_2\|\tilde{\bar L}\|+b_3\|\bs c_1\|+\|\bs c_2\|) \nonumber \\
            & \leq - 2 (1-\epsilon)\bar{\bs p}^\top \bar R \bar P \bar{B} \bar F^* \bar{\bs p} 
            - \bar k_2\|\bar{\bs p}\|^3_\infty + \rho_1(\|\bs y\|_\infty) \nonumber\\
            & \leq - 2 (1-\epsilon)\bar{\bs p}^\top \bar R \bar P \bar{B} \bar F^* \bar{\bs p},
\end{align}
for ${\|\bar{\bs p}\|_\infty \geq \bar \rho_2(\|\bs y\|_\infty)}= : \sqrt[3]{\frac{\rho_1(\|\bs y\|_\infty)}{\bar k_2}}$, where $\bs y$ is defined in~\eqref{cascaded_trivial_y}. Here, $\bar k_2>0$ is an appropriate constant,  and $\bar \rho_1$ and $\bar \rho_2$ are suitable class $\mc K$ functions. Since $ 2k_1(b_1\|\tilde{\bar F}\|+b_2\|\tilde{\bar L}\|+b_3\|\bs c_1\|+\|\bs c_2\|)$ is a positive definite function of $\bs y$, existence of $\rho_1$ follows from~\cite[Lemma 4.3]{HKK:02}. 
The rest of the argument follows similar to the proof of Theorem~\ref{thm:stability}. 

The stability of endemic equilibrium can be shown by modifying the Lyapunov analysis of the strongly connected case in the proof of Theorem~\ref{thm:stability} similarly to the proof of the first statement.  
\end{proof}


\section{Optimal intervention strategies for epidemic containment}\label{sec:resource-allocation}
We now study optimal intervention strategies for the SIS model with population dispersal where we allocate resources to make the disease-free equilibrium (DFE) stable. 
We consider two \vs{types} of resources that control infection rate and recovery rate, respectively. The former can change the infection rate such that $\beta_i^\alpha \in [\underline{\beta}_i^\alpha, \bar{\beta}_i^\alpha]$ and the cost to achieve infection rate $\beta_i^\alpha$ at node $i$ and layer $\alpha$ is $f_i^\alpha(\beta_i^\alpha)$. The latter can change the recovery rate such that $\delta_i^\alpha \in [\underline{\delta}_i^\alpha, \bar{\delta}_i^\alpha]$ and the cost to achieve recovery rate $\delta_i^\alpha$ at node $i$ and layer $\alpha$ is $g_i^\alpha(\delta_i^\alpha)$.

The condition for stability of the DFE from Theorem~\ref{thm:stability} is
\begin{equation*}
    \mu (BF^*-D-L^*) \leq 0. 
\end{equation*}

Let $\mathtt{L} = \oplus_{\alpha=1}^m \mathtt{L}^\alpha$, where, for each $i, j \in \until{n}$, \vs{
\[
\mathtt{\ell}^\alpha_{ij} = \begin{cases}
-l^{\alpha*}_{ij}, & i \ne j, \\
\bar{\nu}-\nu^{\alpha}_i, & i =j,
\end{cases}
\]
where $\mathtt{\ell}^\alpha_{ij}$ and $l^{\alpha*}_{ij}$ denote the $(i,j)$-th entries of $\mathtt{L}^\alpha$ and $L^{\alpha*}$, respectively, $L^{\alpha*}$ is the block of $L^*$ associated with layer $\alpha$,}
$\bar{\nu}=\max \setdef{\nu^{\alpha}_i}{i \in \until{n}, \alpha \in \until{m}}$. 
Let $\bar{\Delta}=\max_{i, \alpha} \bar{\delta}_i^\alpha$, 
$\hat{\delta}_{i}^\alpha=\bar{\Delta}+1-\delta_i^\alpha$ and
$\hat D = \operatorname{diag}(\hat{\delta}_{1}^1, \ldots, \hat{\delta}_{n}^m)$. 
It can be seen that 
\[
\vs{\mu}(BF^*-D-L^*)=\vs{\mu}(BF^*+\hat{D}+\mathtt{L})-\bar{\Delta}-1-\bar{\nu},
\]
where $\vs{\mu} (\cdot)$ denotes the spectral abscissa of the matrix. 

Thus, minimizing $\vs{\mu}(BF^*-D-L^*)$ is same as minimizing $\vs{\mu}(BF^*+\hat{D}+\mathtt{L})$. Note that $(BF^*+\hat{D}+\mathtt{L})$ is also an irreducible non-negative matrix.  
We now recall the following result for irreducible non-negative matrices. 
\begin{lemma}[Theorem 10.2, \cite{Bullo-book_Networks}]\label{lem:perrron}
Let $M$ be an irreducible non-negative matrix with its spectral-abscissa $\mu(M)$. Then
\[
\vs{\mu}(M) =\inf\{\lambda \in \real:M\bs u \leq \lambda \bs u, \; \textup{for } \bs u \gg \bs 0\}.
\]
\end{lemma}

Using Lemma~\ref{lem:perrron}, it follows that minimizing $\vs{\mu}(M)$ is equivalent to minimizing $\lambda$ such that
\begin{equation} \label{lambda_ineq}
    \frac{(M \bs u)_i}{\lambda u_i} \leq 1, \quad u_i>0 .
\end{equation}


\begin{assumption}\label{assump:posynomial}
We assume that the cost functions $f_i^\alpha(\beta_i^\alpha)$ and  $g_i^\alpha(\delta_i^\alpha)=:\tilde{g}_i^\alpha(\bar{\Delta}+1-\delta_i^\alpha)=\tilde{g}_i^\alpha(\hat{\delta}_i^\alpha)$ are posynomials~\vs{\cite[Section 4.5]{SB-LV:04}} in $\beta_i^\alpha$ and $\hat{\delta}_i^\alpha$, respectively. 
\end{assumption}


We now pose the optimization problem to maximize the convergence rate for the DFE. 
\begin{align}\label{eq:multi-layer-optimization}
\begin{split}
    \underset{\lambda, u_i,\beta_i,\hat{\delta}_i}{\minimize}\quad &  \lambda \\
    \subject \quad & \frac{1}{\lambda u_i}  \big((BF^*+\hat{D}+\mathtt{L}) \bs u \big)_i \leq 1,\; i \in \until{nm}\\
    & \sum_{i=1}^n \sum_{\alpha=1}^m \vs{f_i^\alpha(\beta_i^\alpha)+\tilde{g}_i^\alpha(\hat{\delta}_i^\alpha)\leq C }\\
   &\vs{\underline{\beta}_i^\alpha \leq \beta_i^\alpha \leq \bar{\beta}_i^\alpha, \; i \in \until{n}, \alpha \in \until{m}}\\
   &\vs{\bar{\Delta}+1-\bar{\delta}_i^\alpha \leq \hat{\delta}_i^\alpha \leq \bar{\Delta}+1-\underline{\delta}_i^\alpha, \; i \in \until{n},}\\
   & \qquad \qquad \qquad \vs{\alpha \in \until{m},}
\end{split}
   \end{align}
  where $C >0$ is the allocation budget.  
Under Assumption~\ref{assump:posynomial}, optimization problem~\eqref{eq:multi-layer-optimization} is a 
geometric program~\cite{SB-LV:04} and can be solved efficiently.

 \section{Numerical Illustrations} \label{Sec: numerical studies}
In this section, we illustrate our results through numerical simulations. 
\vs{For simplicity, in all our simulations, we choose the same recovery and infection rates at any given node $i$ across each layer $\alpha$, i.e., $\beta_i^\alpha = \beta_i$ and $\delta_i^\alpha = \delta_i$, for each $\alpha$. Additionally, we choose functions $f_i^\alpha$ and $g_i^\alpha$ independent of $i$ and $\alpha$ and omit these indices. } 

\subsection{Comparison of deterministic and stochastic models}
We start with a numerical simulation of the  epidemic model with multi-layer population dispersal in which we treat epidemic spread as well as population dispersal as stochastic processes. We take two population dispersal network layers: a complete graph and a line graph with $20$ nodes each. The population dispersal transition rates are selected equally among outgoing neighbors of a node for both graphs with total transition rate $\nu_i^\alpha = 0.1$, for each $i$ and $\alpha$. \vs{For the stochastic simulations, the initial sizes of sub-populations at each node are selected as
\begin{align*}
 & \begin{smallmatrix}
500\times \{7, &     5, &     3, &     1, &     5, &     7, &     8, &     9, &     6, &     7, &     7, &     7, &     5, &     9, &     8, &     6, &    10, &     9, &    10, &     5
\},
\end{smallmatrix} \text{ and } \\
 &\begin{smallmatrix}
500\times \{3, &     3, &    2, &     1, &     2, &     3, &     4, &     4, &    5, &     3, &    3, &     4, &     4, &     3, &     4, &    2, &     6 , &    3, &     3 , &    4
\},
\end{smallmatrix}   
\end{align*}
in layer 1 and layer 2, respectively.}

If the recovery rates, infection rates, and the initial fraction of the infected population are the same for all the nodes, population dispersal does not play any role. Therefore, we have chosen heterogeneous recovery  rates across nodes to elucidate the influence of population dispersal. \vs{The infection rates $\{\delta_1,\ldots, \delta_{20}\}$ are arbitrarily selected as
\[
 \begin{smallmatrix}
\{ 0.3, & 0.22, & 0.21, & 0.25, & 0.3, & 0.21, & 0.23, & 0.24, & 0.21, & 0.22, \\ &  0.25, & 0.21, & 0.3, & 0.28, & 0.22, & 0.26, & 0.21, & 0.3, & 0.28, & 0.25\}.
\end{smallmatrix}
\]
}
The fractions of infected populations for different cases are shown in Fig.~\ref{fig:Stochastic}. The corresponding simulations of the deterministic model as per Proposition \ref{prop:model} are also shown for comparison.
We consider two cases corresponding to the stable disease-free equilibrium and stable endemic equilibrium, respectively. \vs{Fig.~\ref{fig:Stochastic}~(a) corresponds to $\beta_i =0.2$, for each node, i.e., the case $\delta_i \geq \beta_i$ for each $i$, whereas Fig.~\ref{fig:Stochastic}~(c) corresponds to $\beta_i =0.31$, for each node, i.e., the case $\delta_i< \beta_i$ for each $i$.} The results support statements (iii) and (ii) of Corollary \ref{cor:dis-free} and lead to, respectively, the stable disease-free equilibrium and the stable endemic equilibrium.

\begingroup
\centering
\begin{figure}[htb]
\centering

\subfigure[Stable disease-free equilibrium: Stochastic model]{\includegraphics[width=0.23\textwidth, trim={1mm 2mm 9mm 9mm},clip]{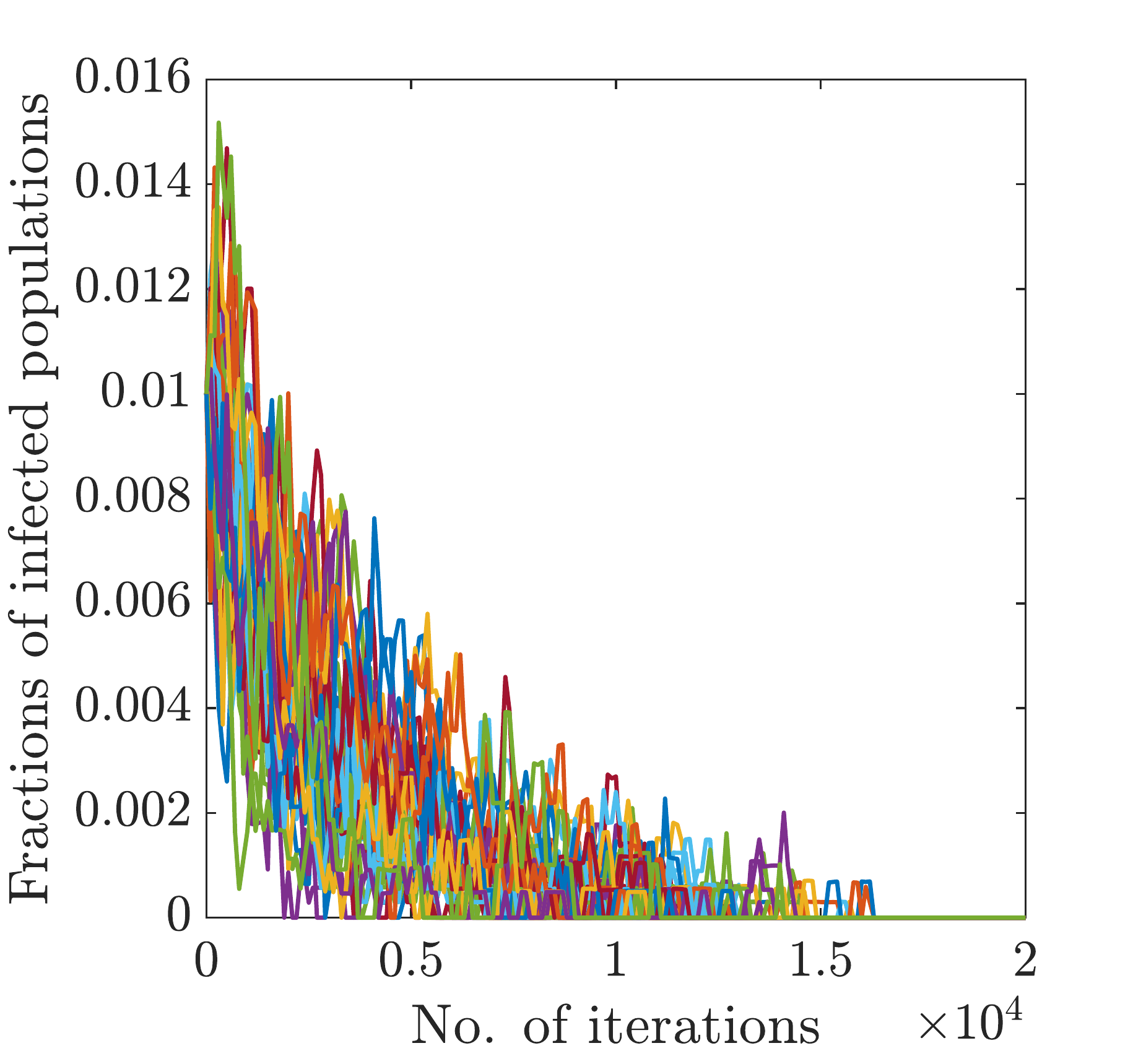}}\label{fig:Stochastic_a}
\subfigure[Stable disease-free equilibrium: Deterministic model]{\includegraphics[width=0.23\textwidth, trim={1mm 2mm 9mm 9mm},clip]{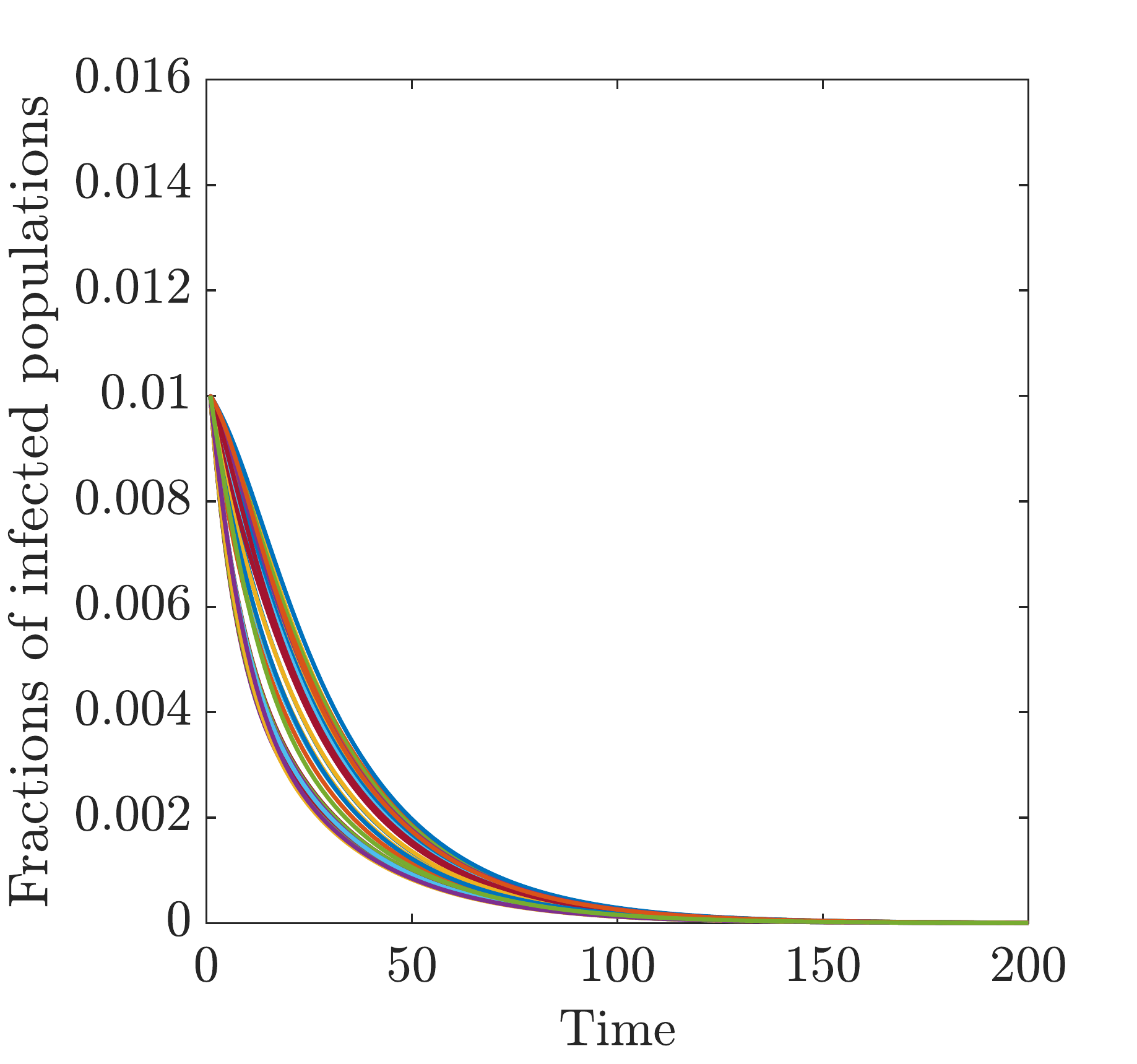}}\label{fig:Stochastic_b}
\subfigure[Stable endemic equilibrium: Stochastic model]{\includegraphics[width=0.23\textwidth, trim={1mm 2mm 9mm 9mm},clip]{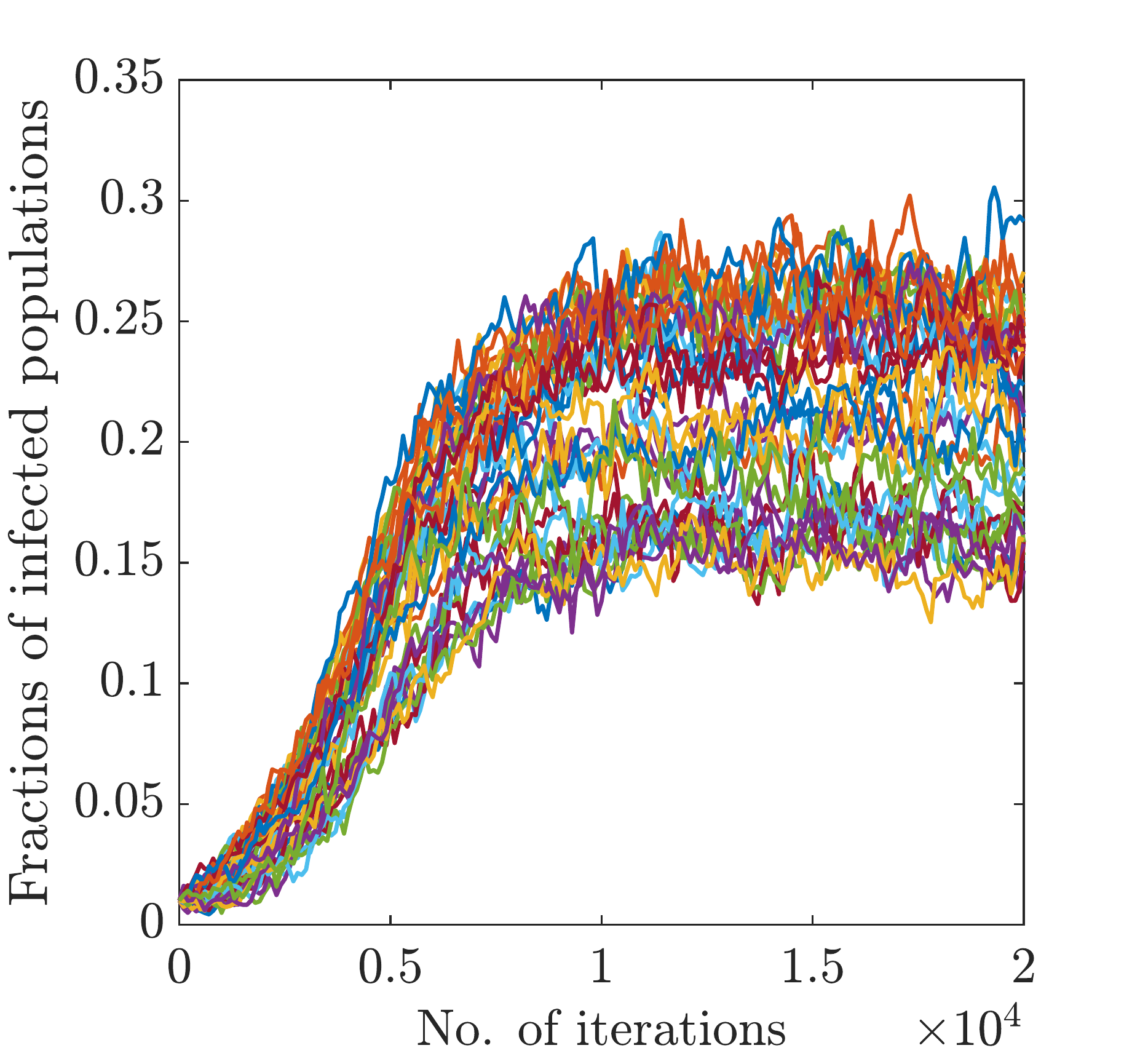}}\label{fig:Stochastic_c}
\subfigure[Stable endemic equilibrium: Deterministic model]{\includegraphics[width=0.23\textwidth, trim={1mm 2mm 5mm 4mm},clip]{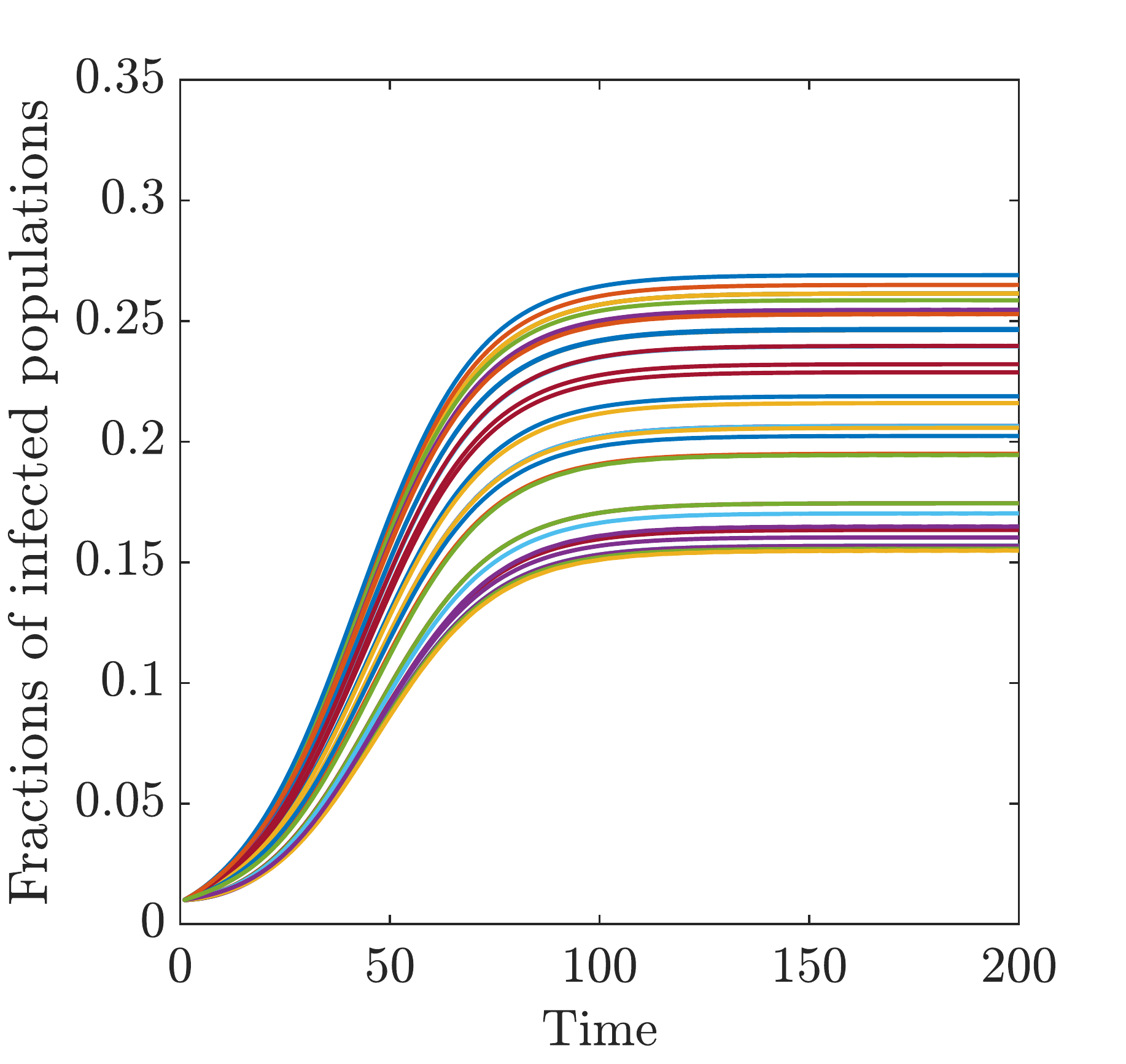}}\label{fig:Stochastic_d}

\caption{\vs{Comparison of the stochastic and deterministic epidemic spread with population dispersal on a two-layer graph with $n=20$ nodes. The layers correspond to complete and line graphs, respectively. 
In each layer, individuals move  to their neighbors  with equal probability with total transition rate $\nu_i^\alpha = 0.1$. Initial fraction of infected population at each node is $p_i^\alpha(0)=0.01$. Each iteration in the stochastic model corresponds to time-step $0.01$ sec. Different trajectories correspond to the fraction of infected population at each node in both layers. Some trajectories are overlapping.} }
\label{fig:Stochastic}
\end{figure}
\endgroup

\subsection{Influence of network structure on epidemic propagation} 

Once we have established the correctness of deterministic model predictions with the stochastic simulations, we perform the simulations using only the deterministic model. We study the effect of multi-layer population dispersal over different pairs of population dispersal graph structures: complete-line graph, complete-ring graph, and complete-star graph. \vs{The number of nodes in each layer is selected as $10$.} We choose different population sizes for the two population dispersal layers and take the population dispersal transition rates to keep the equilibrium distribution of population the same for both the layers across all pairs (taken as uniform equilibrium distribution) by using instantaneous transition rates from the Metropolis-Hastings algorithm \cite{Hastings_MetroplisHastingsMC}. The total transition rate is selected as $\nu_i^\alpha = 0.1$, for each $i$ and $\alpha$. This shows the effect of different population dispersal graph structures on epidemic spread while the equilibrium population distribution remains the same. \vs{We select the infection rate $\beta_i= 0.31$, for each $i$ and the recovery rates $\{\delta_1, \ldots, \delta_{10}\}$ as
\[
 \begin{smallmatrix}
\{ 0.3, &0.22, & 0.21 , & 0.25, & 0.3, & 0.21 , & 0.23 , & 0.24 , &0.21 , & 0.22\}.
\end{smallmatrix}
\]}

Fig.~\ref{fig:Deterministic_SameMobilityEqbDist} shows the fractions of infected population trajectories for $10$ nodes connected with different pairs of graph structures. The values of equilibrium fractions are affected by the presence of population dispersal and are different for different graph structures.

\begingroup
\centering
\begin{figure}[htb]
\centering
\subfigure[Complete-Line graphs; graph 1]{\includegraphics[width=0.23\textwidth, trim={1mm 2mm 9mm 9mm},clip]{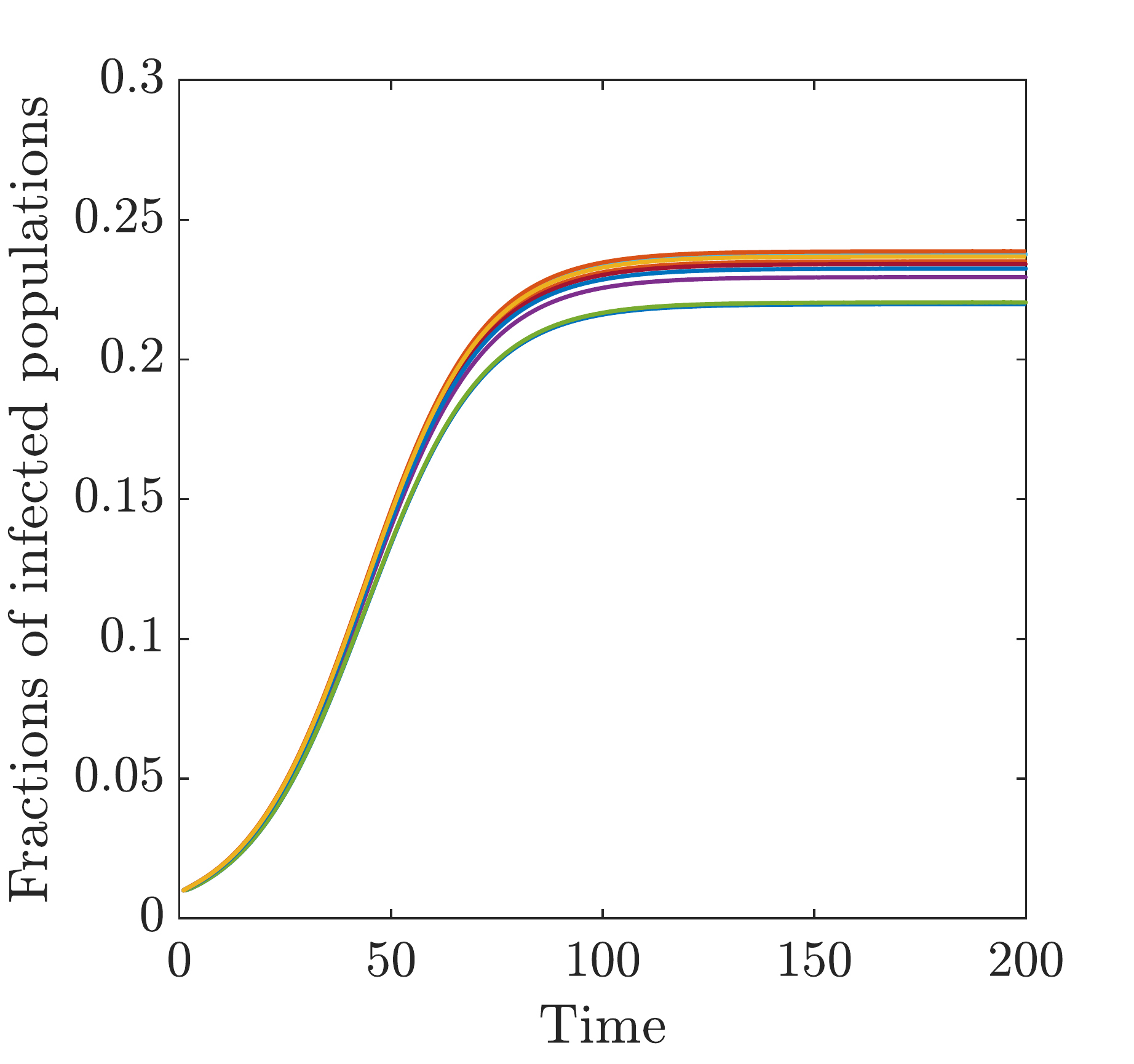}}\label{fig:Samedist_L-L_1}
\subfigure[Complete-Line graphs; graph 2]{\includegraphics[width=0.23\textwidth, trim={1mm 2mm 9mm 9mm},clip]{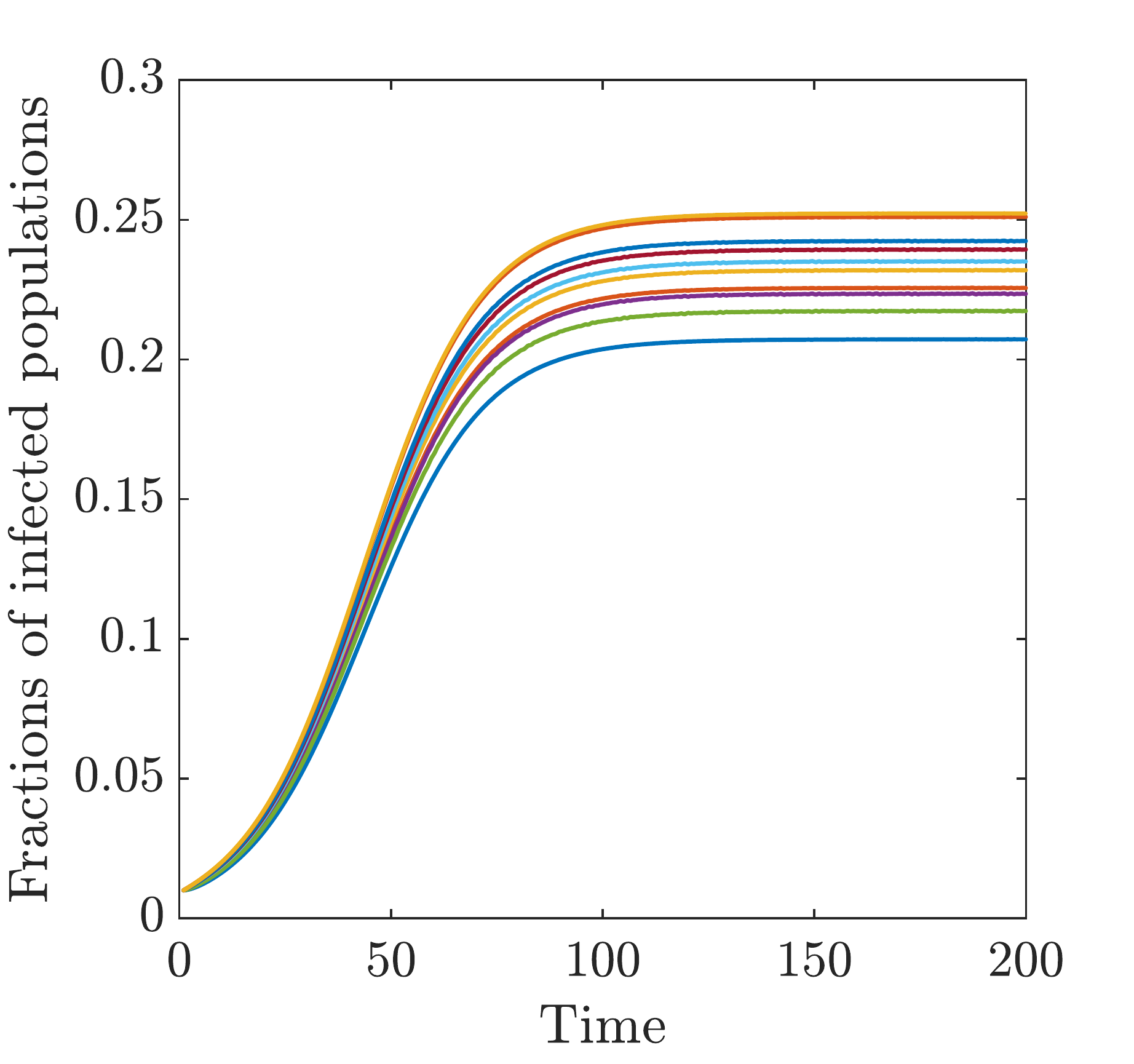}}\label{fig:Samedist_L-L_2}
\subfigure[Complete-Ring graphs; graph 1]{\includegraphics[width=0.23\textwidth, trim={1mm 2mm 9mm 9mm},clip]{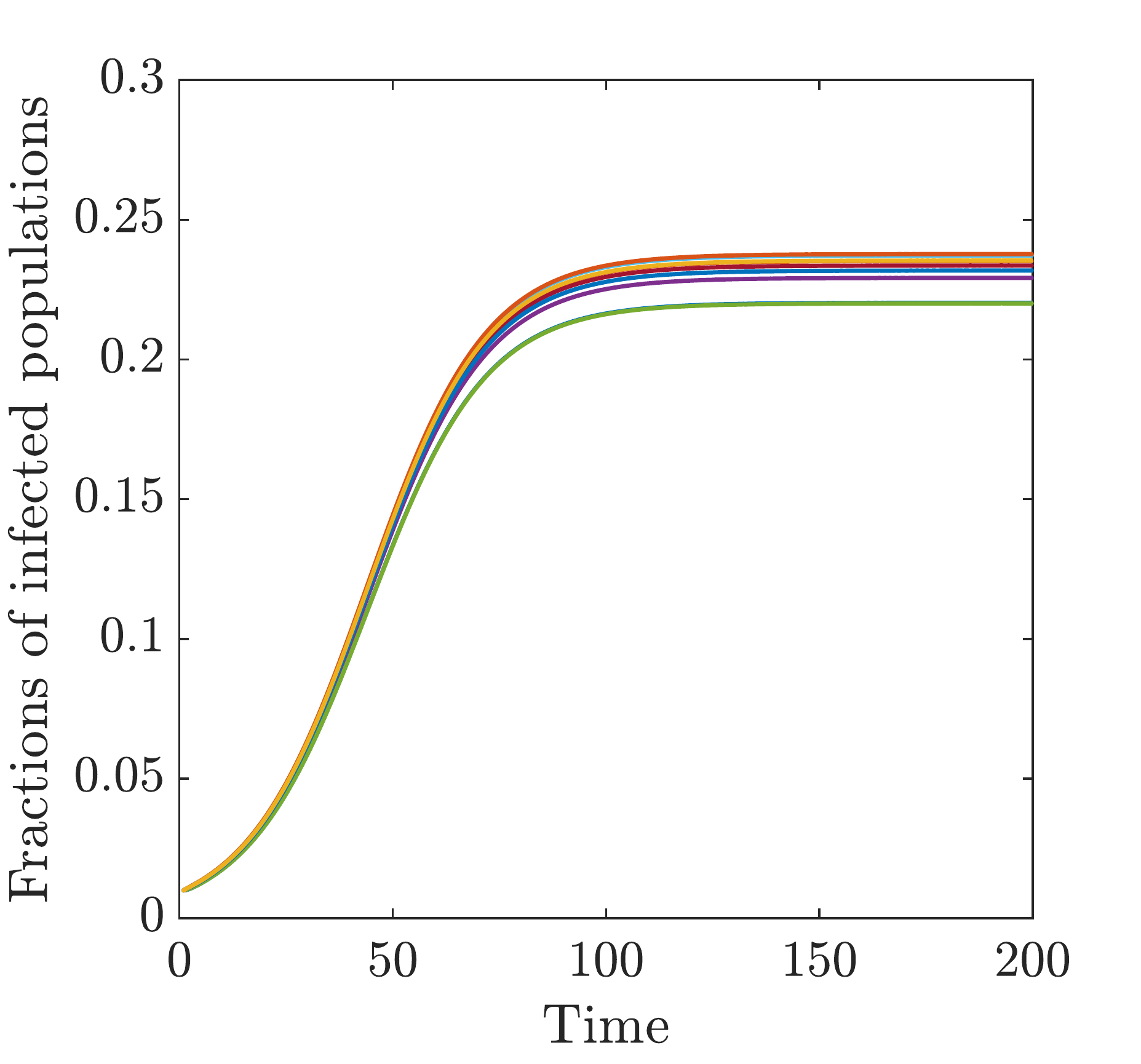}}\label{fig:Samedist_L-R_1}
\subfigure[Complete-Ring graphs; graph 2]{\includegraphics[width=0.23\textwidth, trim={1mm 2mm 9mm 9mm},clip]{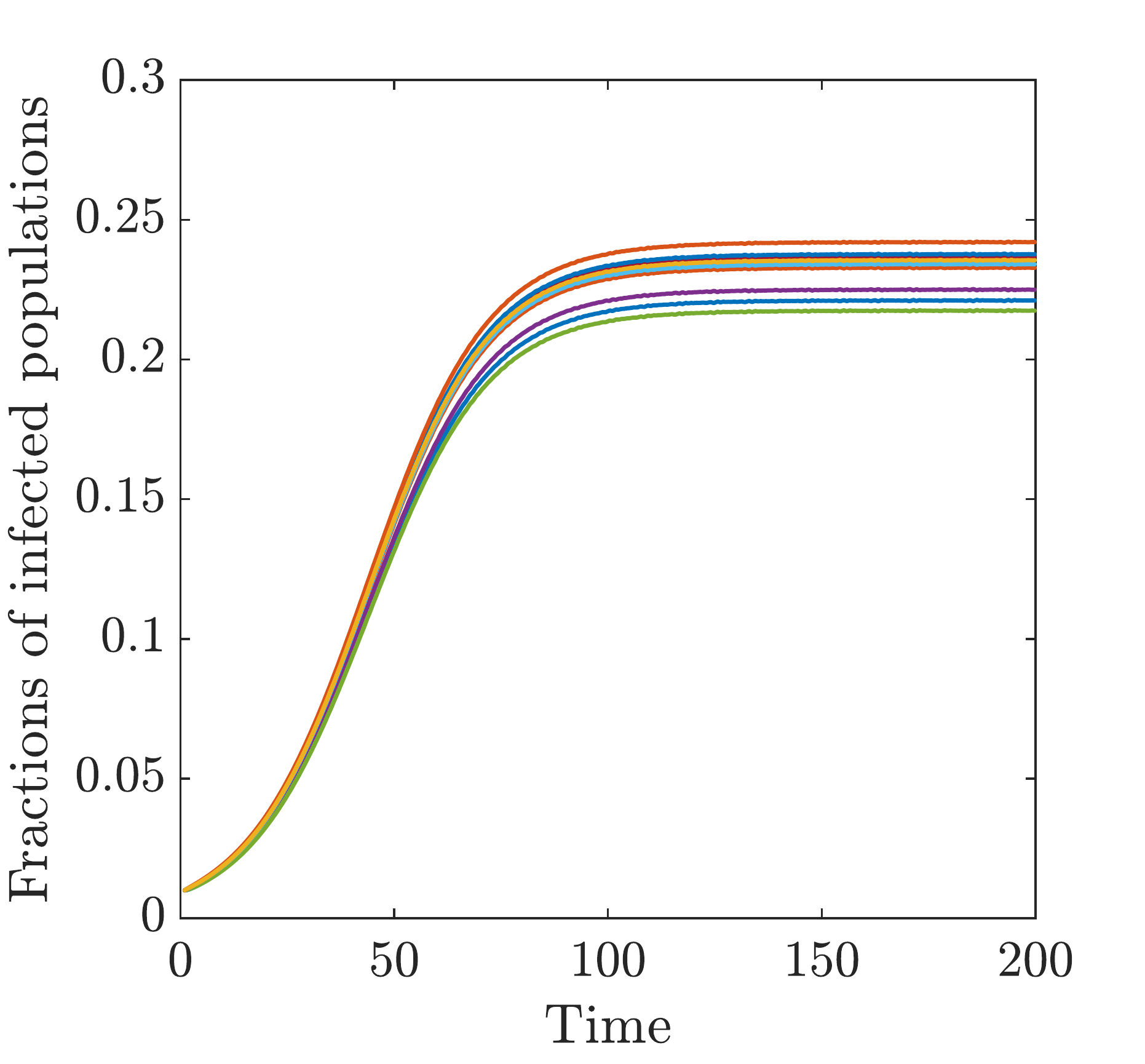}}\label{fig:Samedist_L-R_2}
\subfigure[Complete-Star graphs; graph 1]{\includegraphics[width=0.23\textwidth, trim={1mm 2mm 9mm 9mm},clip]{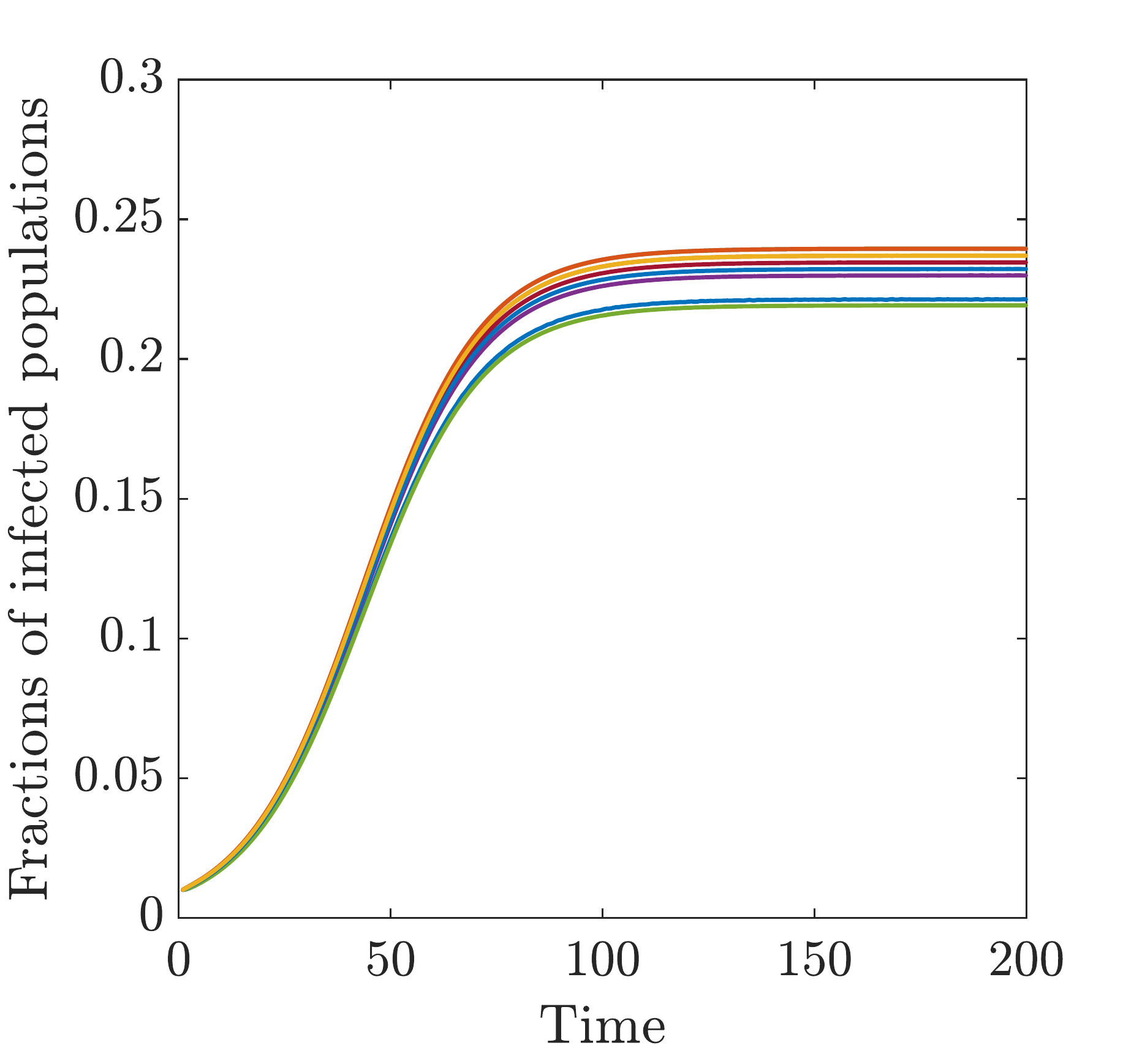}}\label{fig:Samedist_L-S_1}
\subfigure[Complete-Star graphs; graph 2]{\includegraphics[width=0.23\textwidth, trim={1mm 2mm 9mm 9mm},clip]{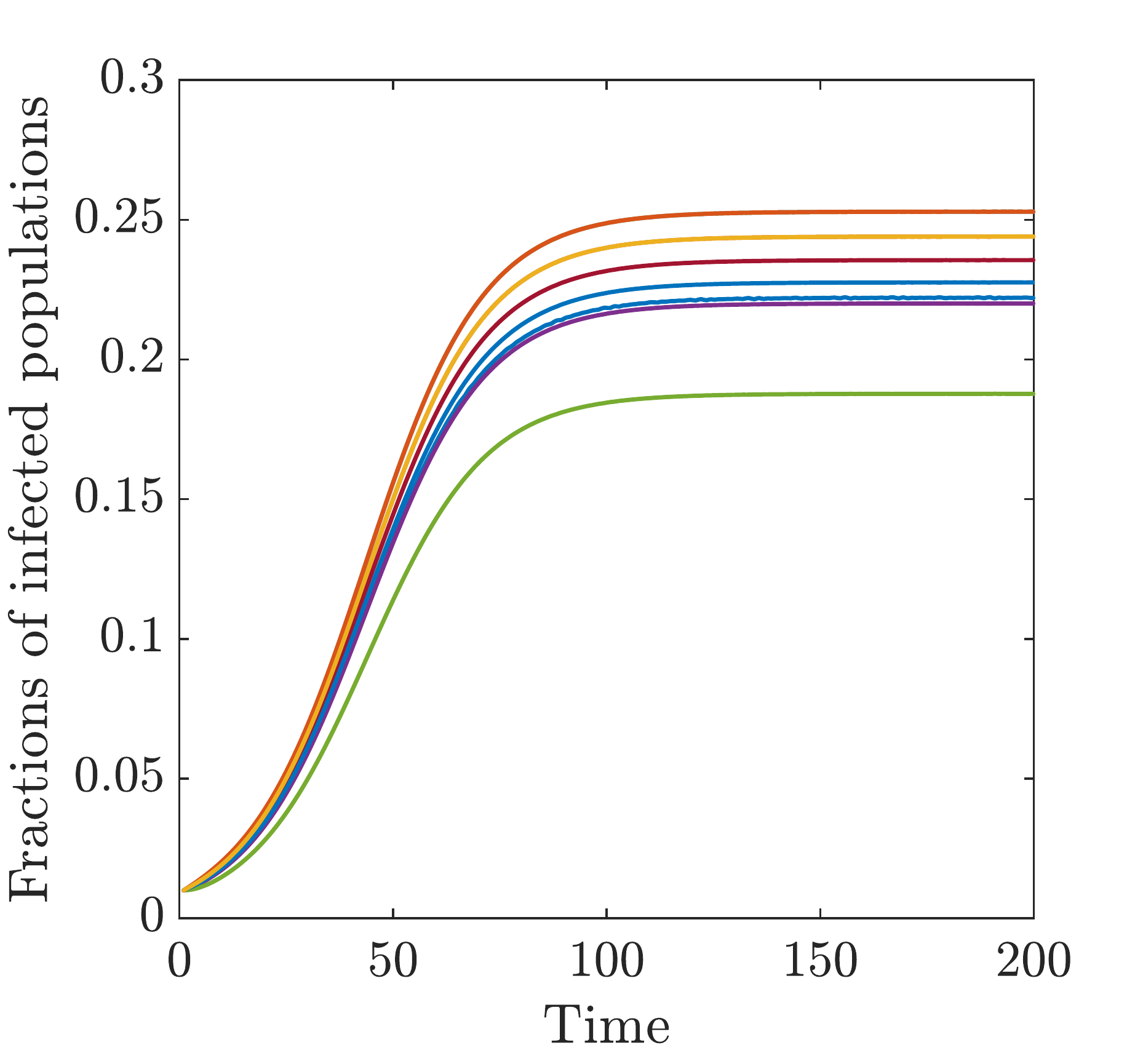}}\label{fig:Samedist_L-S_2}
\caption{\vs{Simulation of the deterministic model of epidemic spread for 2-layer graphs with different structures. Each graph has $10$ nodes and fraction of infected population at each node in both layers is $p_i^\alpha(0)=0.01$.}}
\label{fig:Deterministic_SameMobilityEqbDist}
\end{figure}
\endgroup

Next, we verify Statement (iv) of Corollary \ref{cor:dis-free}, for a single layer population dispersal model, where some recovery rates $\delta_i$'s may be lesser than the infection rates $\beta_i$'s but  the disease-free equilibrium may still be stable. We take a complete graph of $n=20$ nodes and select the population dispersal transition rates equally among outgoing neighbors of a node and the total transition rate $\nu_i = 0.1$, for each node $i$.
These choices fix the values of $\bs w$, $L^*$ and $\lambda_2$. 
Next, we compute $\subscr{s}{lower} = -\frac{\lambda_2}{4n+1}-0.0013$ and take 
$s= 0.8 \subscr{s}{lower}=-0.0010$.
We select $\beta_i=0.3$, $\delta_1=\delta_n=\beta_i+s$ and the rest $\delta_i$'s are assumed to be the same and are computed to satisfy the condition in Statement (iv) of Corollary \ref{cor:dis-free}. This yields $\delta_1 = \delta_n = 0.2990$ and $\delta_i = 0.3099$ for $i \in \{2,\dots,n-1\}$.
Fig.~\ref{fig:Lambda2 sufficient cond Complete graph} shows the trajectories of infected fraction populations. As can be seen, the trajectories converge to the disease-free equilibrium.

\begin{figure}[ht]
    \centering
    \includegraphics[width=0.5\linewidth, trim={1mm 2mm 9mm 9mm},clip]{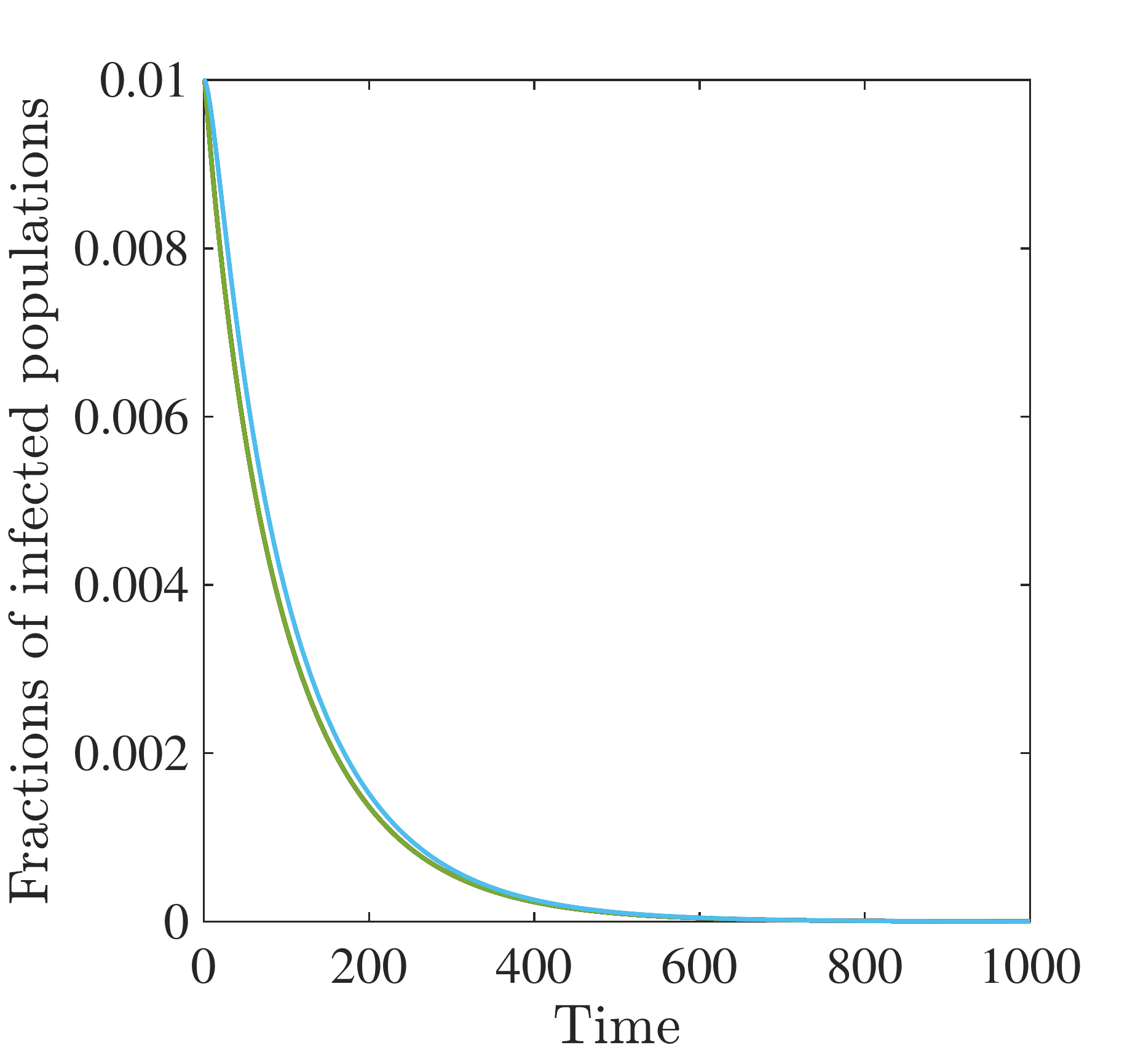}
    \caption{\vs{Illustration of Statement (iv), Corollary \ref{cor:dis-free}. A single-layer complete graph with $20$ nodes is consider with initial fraction of infected population $p_i(0)=0.01$ at each node. The nodes belong of two categories corresponding to deficit and excess of recovery rate, and the two trajectories correspond to each category. The recovery rates are selected to satisfy the  sufficient condition in Statement (iv) of Corollary \ref{cor:dis-free}) for the stability of disease-free equilibrium. }}
    \label{fig:Lambda2 sufficient cond Complete graph}
\end{figure}

\subsection{Optimal intervention for epidemic containment}
We now illustrate budget-constrained optimal intervention strategies in a multi-layer network. We consider two layers: layer 1 as a line graph and layer 2 as a ring graph. We use population dispersal transition rates such that the outgoing rates to neighboring nodes are equal and the total outgoing transition rate at each node is fixed and equal to $0.2$. The total number of individuals in layers 1 and 2 are 300 and 500, respectively. 
We adopt the following cost functions
\begin{equation*}
    \begin{gathered}
    f(\beta_i)=\frac{1}{\beta_i}-\frac{1}{\bar{\beta}_i}, \\ g(\delta_i)=\frac{1}{\bar{\Delta}+1-\delta_i} -\frac{1}{\bar{\Delta}+1-\underline{\delta}_i} =\frac{1}{\hat{\delta}_i} -\frac{1}{\bar{\Delta}+1-\underline{\delta}_i}.
    \end{gathered}
\end{equation*}
The bounds on the infection and recovery rates are $\underline{\beta}_i = 0.1,\quad \bar{\beta}_i = 0.4,\quad\ \underline{\delta}_i=0.1,\quad \bar{\delta}_i = 0.4$.
We compute the optimal infection and recovery rate by solving optimization problem~\eqref{eq:multi-layer-optimization}. 
We also solve the resource allocation problem using a na{\" i}ve strategy in which the budget is allocated equally to each node to modify recovery and infection rates within their bounds. Fig.~\ref{fig:lambda against budget C} shows $\mu(BF^*-D-L^*)$ achieved using the optimal intervention as well as the na{\" i}ve strategy versus the allocation budget $C$. It can be seen that $\mu$ is positive for very low values of budget implying unstable disease-free equilibrium. As the allocation budget $C$ increases, $\mu$ saturates around a value of $-0.3$. Fig.~\ref{fig:cost used} shows the resource used by each strategy versus the allocation budget. Note that the na{\" i}ve strategy uses lesser resources compared to the optimal strategy. This is due to the fact that equal allocation of budget leads to unused resources. 
In particular, the recovery rates reach their upper limit before the infection rates reach their lower limit and subsequently, any additional resource available for recovery remains unused.


\begin{figure}[ht!]
\centering
\subfigure[]{\includegraphics[width=0.475\linewidth, trim={1mm 2mm 9mm 9mm},clip]{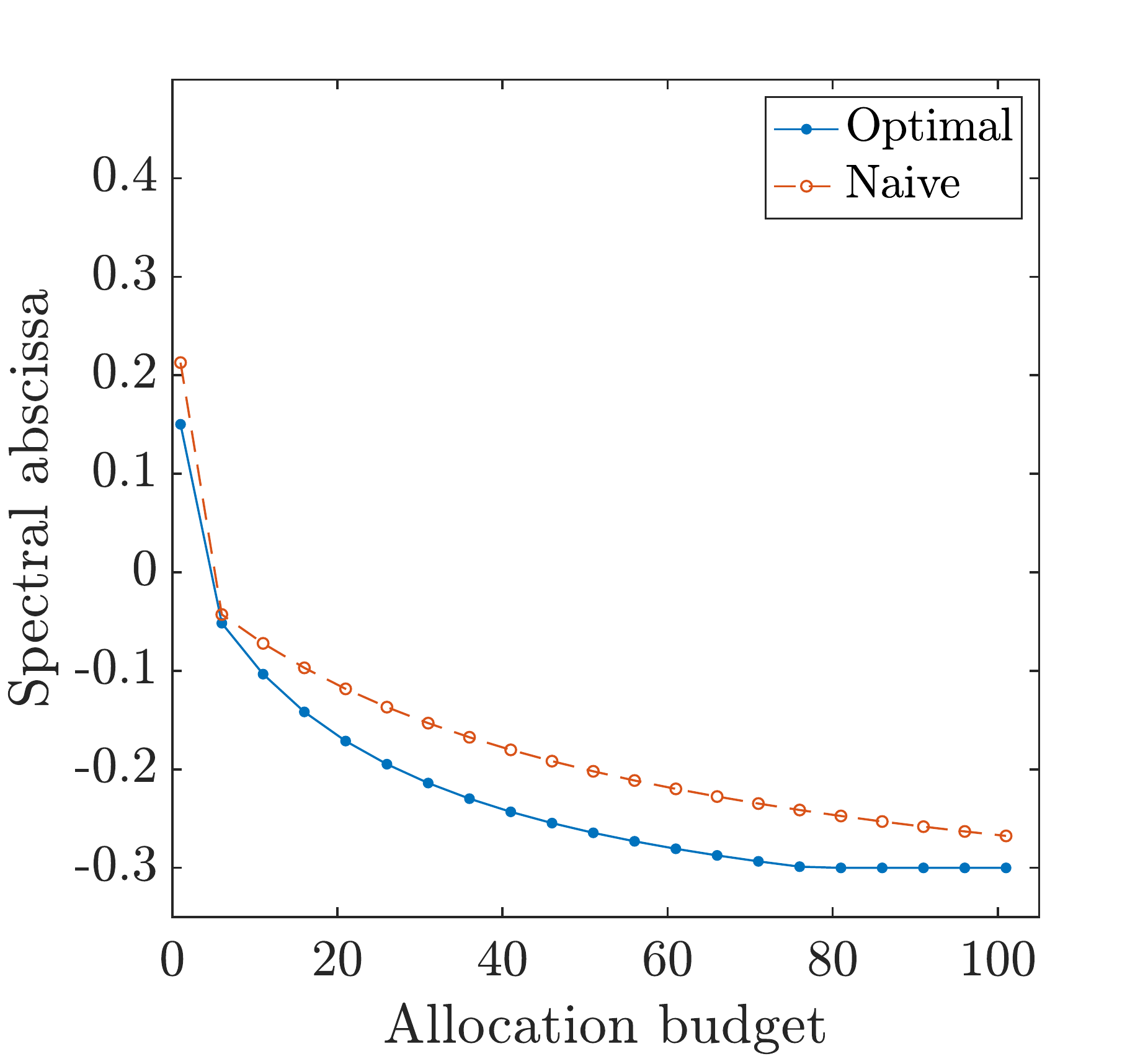}\label{fig:lambda against budget C}}
\subfigure[]{    \includegraphics[width=0.475\linewidth, trim={1mm 2mm 9mm 9mm},clip]{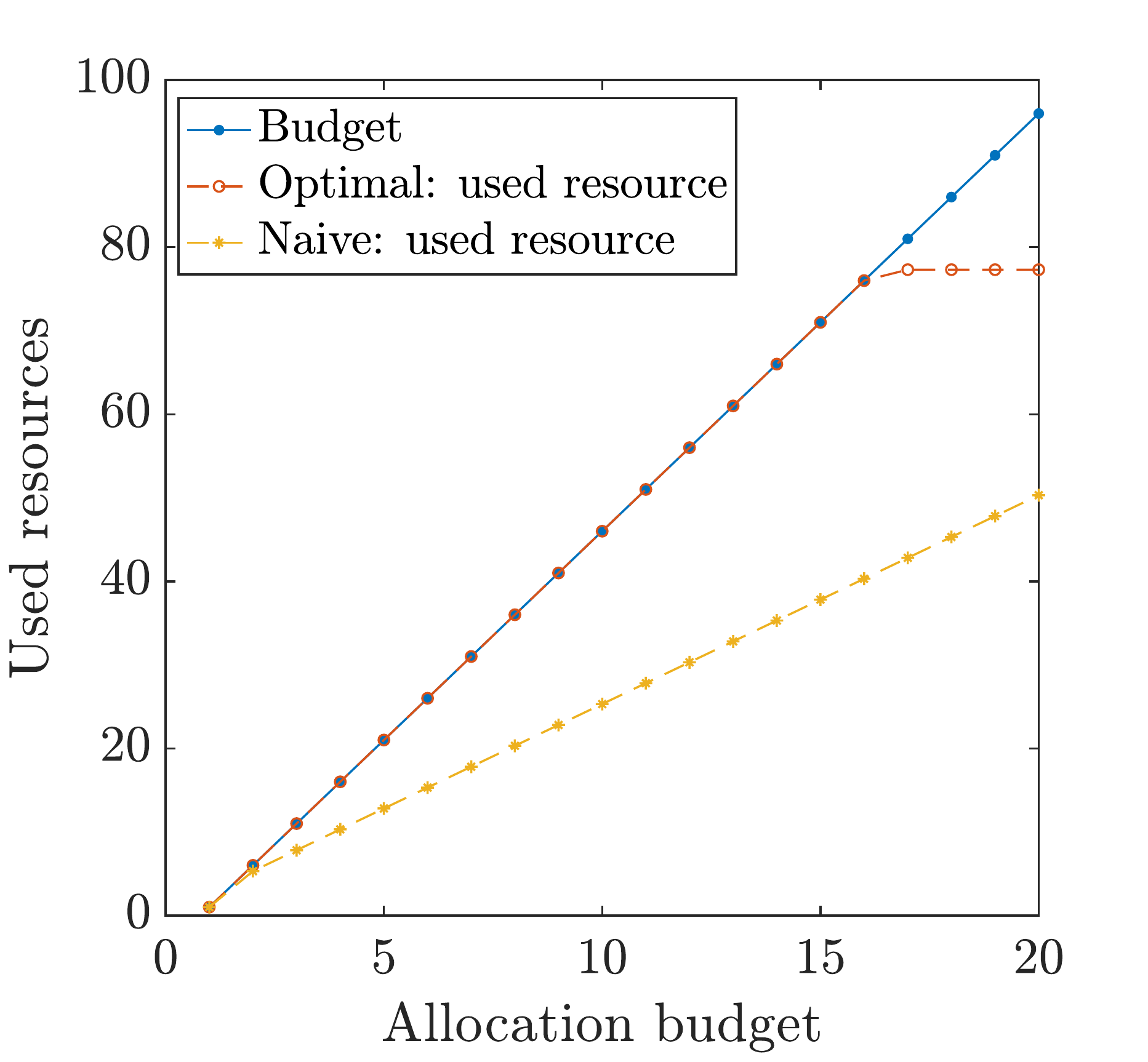} \label{fig:cost used}}
\caption{\vs{Illustration of the optimal intervention strategy for epidemic containment on a two-layer network with $10$ nodes and each layer corresponding to line and ring graph, respectively. The individuals move to neighboring nodes with equal probability with the total transition rate is selected as $\nu_i^\alpha = 0.2$ both layers. (a) The optimal strategy achieves lower spectral abscissa  compared with the n{\" a}ive strategy, for each value of allocated budget $C$. (b) The n{\" a}ive strategy does not use all the allocated budget. }
}
\end{figure}



  
\section{Conclusions and Future Directions} \label{Sec: conclusions}

We derived a continuous-time model for epidemic propagation with population dispersal across a multi-layer network of patches. The epidemic spread within each patch has been modeled as an SIS model with individuals traveling across the patches with different dispersal patterns modeled as layers of a multi-layer network. The derived model has been analyzed to establish the existence and stability of a disease-free equilibrium and an endemic equilibrium under different conditions.
Some necessary and sufficient conditions for the stability of disease-free equilibrium have been established. We studied the optimal intervention for epidemic containment using a geometric program that maximizes the decay rate of disease-free equilibrium under a budget constraint. We also presented numerical studies to support our results and elucidate the effect of population dispersal on epidemic propagation.

\vs{While this work focuses on studying the interaction of population dispersal dynamics with the epidemic spread dynamics, several recent works have focused on the interaction of epidemic dynamics with human behavioral and societal dynamics~\cite{huang2022game, hota2019game,khazaei2021disease,martins2022epidemic,frieswijk2022mean,elokda2021dynamic,satapathi2022coupled,she2022networked,al2021long}. It would be interesting to explore the joint interactions of epidemic, population dispersal, behavioral, and societal dynamics. }


\begin{thebibliography}{10}
	
	\bibitem{VA-VS:19j}
	V.~Abhishek and V.~Srivastava, ``{SIS} epidemic model under mobility on
	multi-layer networks,'' in {\em American Control Conference}, (Denver, CO),
	pp.~3743--3748, July 2020.
	
	\bibitem{anderson1992infectious}
	R.~M. Anderson and R.~M. May, {\em Infectious Diseases of Humans: Dynamics and
		Control}.
	\newblock Oxford University Press, 1992.
	
	\bibitem{DE-JK:10}
	D.~Easley and J.~Kleinberg, {\em Networks, Crowds, and Markets: Reasoning About
		a Highly Connected World}.
	\newblock Cambridge University Press, 2010.
	
	\bibitem{kleinberg2007computing}
	J.~Kleinberg, ``The wireless epidemic,'' {\em Nature}, vol.~449, no.~7160,
	p.~287, 2007.
	
	\bibitem{wang2009understanding}
	P.~Wang, M.~C. Gonz{\'a}lez, C.~A. Hidalgo, and A.-L. Barab{\'a}si,
	``Understanding the spreading patterns of mobile phone viruses,'' {\em
		Science}, vol.~324, no.~5930, pp.~1071--1076, 2009.
	
	\bibitem{zhang2007performance}
	X.~Zhang, G.~Neglia, J.~Kurose, and D.~Towsley, ``Performance modeling of
	epidemic routing,'' {\em Computer Networks}, vol.~51, no.~10, pp.~2867--2891,
	2007.
	
	\bibitem{jin2013epidemiological}
	F.~Jin, E.~Dougherty, P.~Saraf, Y.~Cao, and N.~Ramakrishnan, ``Epidemiological
	modeling of news and rumors on twitter,'' in {\em Proceedings of the Workshop
		on Social Network Mining and Analysis}, p.~8, ACM, 2013.
	
	\bibitem{hethcote2000mathematics}
	H.~W. Hethcote, ``The mathematics of infectious diseases,'' {\em SIAM Review},
	vol.~42, no.~4, pp.~599--653, 2000.
	
	\bibitem{lajmanovich1976deterministic}
	A.~Lajmanovich and J.~A. Yorke, ``A deterministic model for gonorrhea in a
	nonhomogeneous population,'' {\em Mathematical Biosciences}, vol.~28,
	no.~3-4, pp.~221--236, 1976.
	
	\bibitem{ganesh2005effect}
	A.~Ganesh, L.~Massouli{\'e}, and D.~Towsley, ``The effect of network topology
	on the spread of epidemics,'' in {\em Proceedings IEEE 24th Annual Joint
		Conference of the IEEE Computer and Communications Societies.}, vol.~2,
	pp.~1455--1466, IEEE, 2005.
	
	\bibitem{AnalysisandControlofEpidemics_ControlSysMagazine_Pappas}
	C.~{Nowzari}, V.~M. {Preciado}, and G.~J. {Pappas}, ``Analysis and control of
	epidemics: A survey of spreading processes on complex networks,'' {\em IEEE
		Control Systems Magazine}, vol.~36, pp.~26--46, Feb 2016.
	
	\bibitem{meiBullo2017ReviewPaper_DeterministicEpidemicNetworks}
	W.~Mei, S.~Mohagheghi, S.~Zampieri, and F.~Bullo, ``On the dynamics of
	deterministic epidemic propagation over networks,'' {\em Annual Reviews in
		Control}, vol.~44, pp.~116--128, 2017.
	
	\bibitem{zino2021analysis}
	L.~Zino and M.~Cao, ``Analysis, prediction, and control of epidemics: A survey
	from scalar to dynamic network models,'' {\em IEEE Circuits and Systems
		Magazine}, vol.~21, no.~4, pp.~4--23, 2021.
	
	\bibitem{wang2004epidemic}
	W.~Wang and X.-Q. Zhao, ``An epidemic model in a patchy environment,'' {\em
		Mathematical Biosciences}, vol.~190, no.~1, pp.~97--112, 2004.
	
	\bibitem{jin2005effect}
	Y.~Jin and W.~Wang, ``The effect of population dispersal on the spread of a
	disease,'' {\em Journal of Mathematical Analysis and Applications}, vol.~308,
	no.~1, pp.~343--364, 2005.
	
	\bibitem{li2009global}
	M.~Y. Li and Z.~Shuai, ``Global stability of an epidemic model in a patchy
	environment,'' {\em Canadian Applied Mathematics Quarterly}, vol.~17, no.~1,
	pp.~175--187, 2009.
	
	\bibitem{arino2005multi}
	J.~Arino, J.~R. Davis, D.~Hartley, R.~Jordan, J.~M. Miller, and P.~Van
	Den~Driessche, ``A multi-species epidemic model with spatial dynamics,'' {\em
		Mathematical Medicine and Biology}, vol.~22, no.~2, pp.~129--142, 2005.
	
	\bibitem{mesbahi2010graph}
	M.~Mesbahi and M.~Egerstedt, {\em Graph Theoretic Methods in Multiagent
		Networks}, vol.~33.
	\newblock Princeton University Press, 2010.
	
	\bibitem{gomez2015abruptTransitionsSIRI}
	J.~G{\'o}mez-Garde{\~n}es, A.~S. de~Barros, S.~T. Pinho, and R.~F. Andrade,
	``Abrupt transitions from reinfections in social contagions,'' {\em
		Europhysics Letters}, vol.~110, no.~5, p.~58006, 2015.
	
	\bibitem{pagliara_NaomiL2018bistability}
	R.~Pagliara, B.~Dey, and N.~E. Leonard, ``Bistability and resurgent epidemics
	in reinfection models,'' {\em IEEE Control Systems Letters}, vol.~2, no.~2,
	pp.~290--295, 2018.
	
	\bibitem{fall2007epidemiological}
	A.~Fall, A.~Iggidr, G.~Sallet, and J.-J. Tewa, ``Epidemiological models and
	{L}yapunov functions,'' {\em Mathematical Modelling of Natural Phenomena},
	vol.~2, no.~1, pp.~62--83, 2007.
	
	\bibitem{khanafer_Basar2016stabilityEpidemicDirectedGraph}
	A.~Khanafer, T.~Ba{\c{s}}ar, and B.~Gharesifard, ``Stability of epidemic models
	over directed graphs: A positive systems approach,'' {\em Automatica},
	vol.~74, pp.~126--134, 2016.
	
	\bibitem{pagliara2020adaptive}
	R.~Pagliara and N.~E. Leonard, ``Adaptive susceptibility and heterogeneity in
	contagion models on networks,'' {\em IEEE Transactions on Automatic Control},
	vol.~66, no.~2, pp.~581--594, 2020.
	
	\bibitem{pare2020modeling}
	P.~E. Par{\'e}, C.~L. Beck, and T.~Ba{\c{s}}ar, ``Modeling, estimation, and
	analysis of epidemics over networks: An overview,'' {\em Annual Reviews in
		Control}, vol.~50, pp.~345--360, 2020.
	
	\bibitem{bokharaie2010_EpidemicTVnetwork}
	V.~Bokharaie, O.~Mason, and F.~Wirth, ``Spread of epidemics in time-dependent
	networks,'' in {\em International Symposium on Mathematical Theory of
		Networks and Systems}, (Budapest, Hungary), pp.~1717--1719, July 2010.
	
	\bibitem{Preciado2016_EpidemicTVnetwork}
	M.~{Ogura} and V.~M. {Preciado}, ``Stability of spreading processes over
	time-varying large-scale networks,'' {\em IEEE Transactions on Network
		Science and Engineering}, vol.~3, pp.~44--57, Jan 2016.
	
	\bibitem{Beck2018_EpidemicTimeVaryingNetwork}
	P.~E. {Paré}, C.~L. {Beck}, and A.~Nedi{\'c}, ``Epidemic processes over
	time-varying networks,'' {\em IEEE Transactions on Control of Network
		Systems}, vol.~5, pp.~1322--1334, Sep. 2018.
	
	\bibitem{vespignani2012modelling}
	A.~Vespignani, ``Modelling dynamical processes in complex socio-technical
	systems,'' {\em Nature Physics}, vol.~8, no.~1, pp.~32--39, 2012.
	
	\bibitem{hota2021impacts}
	A.~R. Hota, T.~Sneh, and K.~Gupta, ``Impacts of game-theoretic activation on
	epidemic spread over dynamical networks,'' {\em SIAM Journal on Control and
		Optimization}, vol.~60, no.~2, pp.~S92--S118, 2021.
	
	\bibitem{colizza2008epidemicReaction-DiffusionMetapopuln}
	V.~Colizza and A.~Vespignani, ``Epidemic modeling in metapopulation systems
	with heterogeneous coupling pattern: Theory and simulations,'' {\em Journal
		of Theoretical Biology}, vol.~251, no.~3, pp.~450--467, 2008.
	
	\bibitem{Saldana2008continoustime_Reaction-DiffnMetapopln}
	J.~Salda\~na, ``Continuous-time formulation of reaction-diffusion processes on
	heterogeneous metapopulations,'' {\em Phys. Rev. E}, vol.~78, p.~012902, Jul
	2008.
	
	\bibitem{lewien2019time}
	P.~Lewien and A.~Chapman, ``Time-scale separation on networks for multi-city
	epidemics,'' in {\em 2019 IEEE 58th Conference on Decision and Control
		(CDC)}, pp.~746--751, IEEE, 2019.
	
	\bibitem{arino2003multi}
	J.~Arino and P.~Van~den Driessche, ``A multi-city epidemic model,'' {\em
		Mathematical Population Studies}, vol.~10, no.~3, pp.~175--193, 2003.
	
	\bibitem{butler2021effect}
	B.~Butler, C.~Zhang, I.~Walter, N.~Nair, R.~Stern, and P.~E. Par{\'e}, ``The
	effect of population flow on epidemic spread: Analysis and control,'' {\em
		arXiv preprint arXiv:2104.07600}, 2021.
	
	\bibitem{possieri2019mathematical}
	C.~Possieri and A.~Rizzo, ``A mathematical framework for modeling propagation
	of infectious diseases with mobile individuals,'' in {\em 2019 IEEE 58th
		Conference on Decision and Control (CDC)}, pp.~3750--3755, IEEE, 2019.
	
	\bibitem{bichara2015sis}
	D.~Bichara, Y.~Kang, C.~Castillo-Chavez, R.~Horan, and C.~Perrings, ``{SIS} and
	{SIR} epidemic models under virtual dispersal,'' {\em Bulletin of
		Mathematical Biology}, vol.~77, no.~11, pp.~2004--2034, 2015.
	
	\bibitem{ye2020network}
	M.~Ye, J.~Liu, C.~Cenedese, Z.~Sun, and M.~Cao, ``A network {SIS}
	meta-population model with transportation flow,'' {\em IFAC-PapersOnLine},
	vol.~53, no.~2, pp.~2562--2567, 2020.
	
	\bibitem{sahneh2014competitive}
	F.~D. Sahneh and C.~Scoglio, ``Competitive epidemic spreading over arbitrary
	multilayer networks,'' {\em Physical Review E}, vol.~89, no.~6, p.~062817,
	2014.
	
	\bibitem{gracy2022modeling}
	S.~Gracy, P.~E. Par{\'e}, J.~Liu, H.~Sandberg, C.~L. Beck, K.~H. Johansson, and
	T.~Ba{\c{s}}ar, ``Modeling and analysis of a coupled sis bi-virus model,''
	{\em arXiv preprint arXiv:2207.11414}, 2022.
	
	\bibitem{ye2022convergence}
	M.~Ye, B.~D. Anderson, and J.~Liu, ``Convergence and equilibria analysis of a
	networked bivirus epidemic model,'' {\em SIAM Journal on Control and
		Optimization}, vol.~60, no.~2, pp.~S323--S346, 2022.
	
	\bibitem{soriano2018spreading_MultiplexMobilityNetwork_Metapopln}
	D.~Soriano-Pa{\~n}os, L.~Lotero, A.~Arenas, and J.~G{\'o}mez-Garde{\~n}es,
	``Spreading processes in multiplex metapopulations containing different
	mobility networks,'' {\em Physical Review X}, vol.~8, no.~3, p.~031039, 2018.
	
	\bibitem{VA-VS:20e}
	V.~Abhishek and V.~Srivastava, ``{SIR} epidemic model under mobility on
	multi-layer networks,'' in {\em 3rd IFAC Workshop on Cyber-Physical \& Human
		Systems}, (Beijing, China), pp.~803--806, 2020.
	
	\bibitem{preciado2013convex}
	V.~M. Preciado, F.~D. Sahneh, and C.~Scoglio, ``A convex framework for optimal
	investment on disease awareness in social networks,'' in {\em 2013 IEEE
		Global Conference on Signal and Information Processing}, pp.~851--854, IEEE,
	2013.
	
	\bibitem{preciado2013traffic}
	V.~M. Preciado and M.~Zargham, ``Traffic optimization to control epidemic
	outbreaks in metapopulation models,'' in {\em 2013 IEEE Global Conference on
		Signal and Information Processing}, pp.~847--850, IEEE, 2013.
	
	\bibitem{drakopoulos2014efficient}
	K.~Drakopoulos, A.~Ozdaglar, and J.~N. Tsitsiklis, ``An efficient curing policy
	for epidemics on graphs,'' {\em IEEE Transactions on Network Science and
		Engineering}, vol.~1, no.~2, pp.~67--75, 2014.
	
	\bibitem{wan2007network}
	Y.~Wan, S.~Roy, and A.~Saberi, ``Network design problems for controlling virus
	spread,'' in {\em 2007 46th IEEE Conference on Decision and Control},
	pp.~3925--3932, IEEE, 2007.
	
	\bibitem{nowzari2015optimal}
	C.~Nowzari, V.~M. Preciado, and G.~J. Pappas, ``Optimal resource allocation for
	control of networked epidemic models,'' {\em IEEE Transactions on Control of
		Network Systems}, vol.~4, no.~2, pp.~159--169, 2015.
	
	\bibitem{khanafer2014information}
	A.~Khanafer and T.~Ba{\c{s}}ar, ``Information spread in networks: Control,
	games, and equilibria,'' in {\em 2014 Information Theory and Applications
		Workshop (ITA)}, pp.~1--10, IEEE, 2014.
	
	\bibitem{watkins2019robust}
	N.~J. Watkins, C.~Nowzari, and G.~J. Pappas, ``Robust economic model predictive
	control of continuous-time epidemic processes,'' {\em IEEE Transactions on
		Automatic Control}, vol.~65, no.~3, pp.~1116--1131, 2019.
	
	\bibitem{wang2022state}
	Y.~Wang, S.~Gracy, C.~A. Uribe, H.~Ishii, and K.~H. Johansson, ``A state
	feedback controller for mitigation of continuous-time networked sis
	epidemics,'' {\em arXiv preprint arXiv:2210.04169}, 2022.
	
	\bibitem{Bullo-book_Networks}
	F.~Bullo, {\em Lectures on Network Systems}.
	\newblock Kindle Direct Publishing, 1.4~ed., 2020.
	\newblock With contributions by J. Cortes, F. Dorfler, and S. Martinez.
	
	\bibitem{HKK:02}
	H.~K. Khalil, {\em Nonlinear Systems}.
	\newblock Prentice Hall, third~ed., 2002.
	
	\bibitem{slotine1991applied}
	J.-J.~E. Slotine and W.~Li, {\em Applied Nonlinear Control}.
	\newblock Prentice Hall Englewood Cliffs, NJ, 1991.
	
	\bibitem{wu2005bounds}
	C.~W. Wu, ``On bounds of extremal eigenvalues of irreducible and $m$-reducible
	matrices,'' {\em Linear Algebra and its Applications}, vol.~402, pp.~29--45,
	2005.
	
	\bibitem{SB-LV:04}
	S.~Boyd and L.~Vandenberghe, {\em Convex Optimization}.
	\newblock Cambridge University Press, 2004.
	
	\bibitem{Hastings_MetroplisHastingsMC}
	W.~K. Hastings, ``{Monte Carlo} sampling methods using {Markov} chains and
	their applications,'' {\em Biometrika}, vol.~57, no.~1, pp.~97--109, 1970.
	
	\bibitem{huang2022game}
	Y.~Huang and Q.~Zhu, ``Game-theoretic frameworks for epidemic spreading and
	human decision-making: A review,'' {\em Dynamic Games and Applications},
	pp.~1--42, 2022.
	
	\bibitem{hota2019game}
	A.~R. Hota and S.~Sundaram, ``Game-theoretic vaccination against networked sis
	epidemics and impacts of human decision-making,'' {\em IEEE Transactions on
		Control of Network Systems}, vol.~6, no.~4, pp.~1461--1472, 2019.
	
	\bibitem{khazaei2021disease}
	H.~Khazaei, K.~Paarporn, A.~Garcia, and C.~Eksin, ``Disease spread coupled with
	evolutionary social distancing dynamics can lead to growing oscillations,''
	in {\em 2021 60th IEEE Conference on Decision and Control (CDC)},
	pp.~4280--4286, IEEE, 2021.
	
	\bibitem{martins2022epidemic}
	N.~C. Martins, J.~Certorio, and R.~J. La, ``Epidemic population games and
	evolutionary dynamics,'' {\em arXiv preprint arXiv:2201.10529}, 2022.
	
	\bibitem{frieswijk2022mean}
	K.~Frieswijk, L.~Zino, M.~Ye, A.~Rizzo, and M.~Cao, ``A mean-field analysis of
	a network behavioral--epidemic model,'' {\em IEEE Control Systems Letters},
	vol.~6, pp.~2533--2538, 2022.
	
	\bibitem{elokda2021dynamic}
	E.~Elokda, S.~Bolognani, and A.~R. Hota, ``A dynamic population model of
	strategic interaction and migration under epidemic risk,'' in {\em 2021 60th
		IEEE Conference on Decision and Control (CDC)}, pp.~2085--2091, IEEE, 2021.
	
	\bibitem{satapathi2022coupled}
	A.~Satapathi, N.~K. Dhar, A.~R. Hota, and V.~Srivastava, ``Coupled evolutionary
	behavioral and disease dynamics under reinfection risk,'' {\em arXiv preprint
		arXiv:2209.07348}, 2022.
	
	\bibitem{she2022networked}
	B.~She, J.~Liu, S.~Sundaram, and P.~E. Par{\'e}, ``On a networked sis epidemic
	model with cooperative and antagonistic opinion dynamics,'' {\em IEEE
		Transactions on Control of Network Systems}, 2022.
	
	\bibitem{al2021long}
	M.~A. Al-Radhawi, M.~Sadeghi, and E.~D. Sontag, ``Long-term regulation of
	prolonged epidemic outbreaks in large populations via adaptive control: a
	singular perturbation approach,'' {\em IEEE Control Systems Letters}, vol.~6,
	pp.~578--583, 2021.
	
	\bibitem{berman1994nonnegative}
	A.~Berman and R.~J. Plemmons, {\em Nonnegative Matrices in the Mathematical
		Sciences}, vol.~9.
	\newblock SIAM, 1994.
	
	\bibitem{VA-VS:19c}
	V.~Abhishek and V.~Srivastava, ``On epidemic spreading under mobility on
	networks,'' Tech. Rep. arXiv:1909.02647, arXiv e-print, 2019.
	
\end{thebibliography}

\appendix 

\subsection{Proof of Theorem 1 (iii): Existence of an endemic equilibrium} \label{Appendix: existence of non-trivial eqb}
We first state some properties of M-matrices, which we will use in the proof.

\begin{theorem}[\bit{Properties of M-matrix, \cite{berman1994nonnegative}}]\label{M-matrix properties}
For a real Z-matrix (i.e., a matrix with all off-diagonal terms non-positive) $A\in \real^{n\times n}$, the following statements are equivalent to $A$ being a non-singular M-matrix
\begin{enumerate}
    \item \bit{Stability}: real part of each eigenvalue of $A$ is positive;
    \item \bit{Inverse positivity}: $A^{-1}\geq0$ (for irreducible $A$, $A^{-1}>0$);
    \item \bit{Regular splitting}: $A$ has a convergent regular splitting, i.e., $A$ has a representation $A=M-N$, where $M^{-1}\geq 0$, $N\geq0$ (called regular splitting), with $M^{-1}N$ convergent, i.e., $\rho(M^{-1}N)<1$;
    \item \bit{Convergent regular splitting}: every regular splitting of $A$ is convergent. Further, for a singular M-matrix (i.e. singular Z-matrix with real part of eigenvalues non-negative) regular splitting of $A$ gives $\rho(M^{-1}N)=1$;
    \item \bit{Semi-positivity}: there exists $\bs x \gg \bs 0$ such that $A\bs x \gg \bs 0$;
     \item \bit{Modified semi-positivity}: there exists $\bs x \gg \bs 0$ such that $\bs y = A\bs x > \bs 0$ and matrix $\hat A$ defined by
     \begin{equation*}
         \hat A_{ij}=\begin{cases}
                      1  & \text{if } A_{ij}\neq 0 \text{ or } y_i \neq 0,\\
                      0  & \text{otherwise}
                      \end{cases}
     \end{equation*}
     is irreducible.
    
\end{enumerate}

\end{theorem}
\medskip

An irreducible Laplacian matrix perturbed with a non-negative diagonal matrix with at least one positive element is a non-singular M-matrix (consider Theorem \ref{M-matrix properties} (vi) with $\bs x = \bs 1 \gg \bs 0$). By Assumption~\ref{Assumption:Dnot0}, for each $\alpha$, there exists $k_{\alpha}$ such that $\delta^\alpha_{k_\alpha} > 0$. Hence, block diagonal submatrices of the matrix $L^*+D$ are all non-singular M-matrices, and  $L^*+D$ is a non-singular M-matrix. Similar arguments imply $B^{-1}(L^*+D)$ is a non-singular M-matrix.

We show below that in the case of $\mu (BF^*-D-L^*) > 0$ there exists an endemic equilibrium $\bs p^* \gg 0$, i.e.,
\begin{equation*} 
    \dot{\bs p} |_{\bs p =\bs  p^*} = (BF^*-D-L^*- P^* BF^*) \bs p^* = 0 .
\end{equation*}

Similar to \cite{fall2007epidemiological}, we use Brouwer's fixed point theorem. We split the non-negative matrix $F^*$ as $F^* = I-M$, where $M$ is a Laplacian matrix. Rearranging the terms and writing the above as an equation in $\bs p$ to be satisfied at $\bs p^*$ leads to
\begin{align}
 \!\!\!    (L^*+D)((L^*+D)^{-1} B - I)\bs p 
    & \!= B(P + (I-P)M) \bs p .   \label{eqAppendixeqbM}
\end{align}

Define $A := (L^*+D)^{-1} B$. Since $A^{-1} = B^{-1} (L^*+D)$ is a non-singular M-matrix, its inverse $A$ is non-negative \cite{berman1994nonnegative}. Rearranging  \eqref{eqAppendixeqbM} leads to 
\begin{equation*}
   \bs p = H (\bs p) = (I + A(P+(I-P)M))^{-1}A \bs p .
\end{equation*}

Now we show that $H(\bs p)$ as defined above is a monotonic function in the sense that $\bs p_{2} \geq \bs p_{1}$ implies $H(\bs p_{2}) \geq H(\bs p_{1})$. Define $\tilde{\bs p} := \bs p_2 - \bs p_1$ and $\tilde{P} := \operatorname{diag}(\tilde{\bs p})$. Then,
\begin{align}
  & H(\bs p_{2}) - H(\bs p_{1}) \nonumber \\
         &=
          \left(A^{-1}+P_{2} + (I-P_{2})M\right)^{-1}\bs p_{2} \nonumber \\
& \quad         - \left(A^{-1}+P_{1}+(I-P_{1})M\right)^{-1}\bs p_{1} \nonumber \\
         & = \left(A^{-1}+P_{2} + (I-P_{2})M\right)^{-1}  
         \Big(\bs p_{2} - \nonumber \\
         & \quad \left(A^{-1}+P_{2}+ (I-P_{2})M\right)\left(A^{-1}+P_{1}+(I-P_{1})M\right)^{-1}\bs p_{1}\Big)\nonumber\\
         & = \left(A^{-1}+P_{2} + (I-P_{2})M\right)^{-1} \times \nonumber \\ 
         &\qquad \qquad \Big(\tilde{\bs p} 
        - \tilde{P}(I-M)\left(A^{-1}+P_{1}+(I-P_{1})M\right)^{-1}\bs p_{1}\Big)\nonumber\\
         & = (A^{-1}+P_{2} + (I-P_{2})M)^{-1} \times   \Big(I  - \nonumber \\ 
         & \quad \operatorname{diag}\left((I-M)(A^{-1}+P_{1}+(I-P_{1})M)^{-1}\bs p_{1}\right)\Big)\tilde{\bs p} \label{eqHMultiplex}  
\end{align}

Since $(A^{-1}+P_{2} + (I-P_{2})M) = B^{-1}(L^*+D) + P_{2} +(I-P_{2})M$ is a non-singular M-matrix (consider Theorem \ref{M-matrix properties} (vi) with $\bs x =\bs 1 \gg \bs 0$), its inverse and hence the first term above is non-negative. The second term is shown to be non-negative as below. Rewriting the second term 
\begin{align}
        & \Big(I - \operatorname{diag}\left((I-M)(A^{-1}+P_{1}+(I-P_{1})M)^{-1} P_{1} \bs 1\right)\Big) \nonumber \\
        & = \operatorname{diag}\Big(\left(I - (I-M)(I + A P_{1} + A (I-P_{1})M )^{-1} A P_{1}\right) \bs 1\Big) \nonumber \\
         & = \operatorname{diag}\Big(\big(I - M - (I-M)(I + A P_{1} + A (I-P_{1})M )^{-1} \nonumber \\
         & \quad \left(A P_{1} + A (I-P_{1})M\right)\big) \bs 1\Big)  \nonumber\\
        & = \operatorname{diag}\Big((I-M)\big(I - (I + A P_{1} + A (I-P_{1})M )^{-1} \nonumber \\
        & \quad \left(A P_{1} + A (I-P_{1})M \right)\big) \bs 1\Big) \nonumber \\
         & = \operatorname{diag}\Big((I-M)\left(I + A P_{1} + A (I-P_{1})M \right)^{-1} \bs 1\Big) \nonumber \\
         & = \operatorname{diag}\Big(F^* \big(A^{-1} + P_{1} + (I-P_{1})M \big)^{-1} A^{-1} \bs 1\Big)  \geq 0 , \label{eqIAPMultiplex}
\end{align}
where we have used the identity
\begin{equation*}
    (I + X)^{-1} = I - (I+X)^{-1}X ,
\end{equation*}
and $M \bs 1 = \bs 0$ , since $M$ is a Laplacian matrix. The last inequality in \eqref{eqIAPMultiplex} holds as $A^{-1} \bs 1 = B^{-1} (L^* + D) \bs 1 = B^{-1} D \bs 1 \geq \bs 0$ and $(A^{-1} + P_{1} + (I-P_{1})M)^{-1} \geq \bs 0 $, since it is the inverse of an M-matrix. The last term in the last line of \eqref{eqHMultiplex} is $\tilde{\bs p} \geq \bs 0$. This implies that $H(\bs p)$ is a monotonic function. Also, argument similar to above can be used to show that $H(\bs p) \leq \bs 1$ for all $\bs 0 \leq \bs p \leq \bs 1$. Therefore, $H(\bs 1) \leq \bs 1$.\\

Applying the converse of Theorem \ref{M-matrix properties}~(iv), with Z-matrix as $(L^*+D)-BF^*$, where $(L^*+D)^{-1}\geq0$, $BF^*\geq0$ implies $\mu \left(BF^*-(D+L^*)\right) > 0$ if and only if $R_0= \rho(AF^*) = \rho(A(I-M)) > 1$. Now, $A$ is a block-diagonal matrix with block-diagonal terms as $A^{\alpha} = (L^{*\alpha}+D^{\alpha})^{-1}B^{\alpha}$, which are inverse of irreducible non-singular M-matrices and hence are positive. Using the expression for $F$ gives $AF^*= [(A^1 \subscr{F}{blk}(\bs x^*))^\top,\dots,(A^m \subscr{F}{blk}(\bs x^*))^\top]^\top$. Since $A^\alpha > 0$ and $\subscr{F}{blk}(\bs x^*) \geq 0 $ with no zero column, $AF^*>0$ and hence irreducible.  Since $AF^*$ is an irreducible non-negative matrix, Perron-Frobenius theorem implies $\rho(AF^*)$ is a simple eigenvalue satisfying $AF^* \bs u = \rho (AF^*) \bs u = R_0 \bs u$ with $\bs u \gg \bs 0$. Using $F^*=I-M$ implies
\begin{align*}
    A\bs u &= R_0\bs u + AM\bs u = (R_0 -1)\bs u + (I+AM)\bs u .
\end{align*}

Define $U :=\operatorname{diag}(\bs u)$ and $\gamma := \frac{R_0 -1}{R_0}$. Letting $\bs p = \epsilon \bs u$ , we show that $\exists$ $\epsilon _0$ such that  $\epsilon \in (0, \epsilon_0)$ implies  $H(\epsilon \bs u)\geq \epsilon \bs{u}$. Let
\begin{align*}
       \epsilon K(\epsilon) &:= H(\epsilon \bs u) - \epsilon \bs u \\
        & = \big(I+\epsilon AU + A(I-\epsilon U)M\big)^{-1} A\epsilon \bs u - \epsilon \bs u \\
        & = \epsilon \Big(\big(I+\epsilon AU + A(I-\epsilon U)M\big)^{-1}(R_0 -1) \bs u \\
        & \quad + \big(I+\epsilon AU + A(I-\epsilon U)M\big)^{-1}(I+AM) \bs u - \bs u\Big).
\end{align*}
Now we evaluate $ K(\epsilon)$ at $\epsilon = 0$.
\begin{align*}
       K(0) &= (I+AM)^{-1}(R_0 -1) \bs u \\
        &= (R_0 -1)(A^{-1}+M)^{-1}A^{-1} \bs u \\
        &= \frac{(R_0 -1)}{R_0}(A^{-1}+M)^{-1}F^* \bs u \\
        &= \gamma (A^{-1}+M)^{-1}F^* \bs u  \gg \bs 0 .
\end{align*}

The last inequality follows as $\gamma$ and $\bs u$ are both positive, and $(A^{-1}+M)^{-1}F^* = (B^{-1}(L+D)+M)^{-1} F^*> 0$ as $B^{-1}(L+D)+M$ is an irreducible M-matrix and hence its inverse is positive and $F^* \geq 0$  with no zero column. 
Since $K(\epsilon)$ is a continuous function of $\epsilon$ , $\exists$ $\epsilon_0$ such that $\epsilon _0 > \epsilon >0$ implies $K(\epsilon) \gg \bs 0$ and therefore, $H(\epsilon \bs u)\geq \epsilon \bs{u}$.
Therefore there exists an $\epsilon > 0 $ such that $H(\epsilon \bs u) - \epsilon \bs u \geq \bs 0$ or equivalently, $H(\epsilon \bs u)\geq \epsilon \bs{u}$. Taking the closed compact set $J = [ \epsilon \bs u, \bs 1]$, $H(\bs p): J \to J$ is a continuous function of $\bs p$. Brouwer's fixed point theorem implies there exists a fixed point of $H$ in $J$. This proves the existence of an endemic equilibrium $\bs p^* \gg \bs 0$ when $\mu (BF^*-D-L^*) > 0$ or equivalently $R_0 >1$. The uniqueness is further shown in the following proposition.

\begin{proposition}
{If the mapping $H$ has a strictly positive fixed point, then it is unique.}
\end{proposition}

\begin{proof}
Assume there are two strictly positive fixed points: $\bs 0 \ll \bs p^*\ll \bs 1$ and $\bs 0 \ll \bs q^*\ll \bs 1$. Strict inequality compared to $\bs 1$ is assumed which can be easily proved for any equilibrium point using \eqref{eq_SIS_pi} under Assumptions \ref{Assumption:StrongConnectivity} and \ref{Assumption:Dnot0}.  Define 
\[\eta:=\max \frac{p^*_i}{q^*_i},\quad  k:= \arg\max
\frac{p^*_i}{q^*_i}, \quad z_i = \min (\eta q^*_i,1)\]
Therefore, $\bs p^* \leq \bs z \leq \eta \bs q^*$. Lets assume $\eta >1$, which implies $\bs q^* \ll \bs z $. First we will show that $H(\bs z) \ll \eta H(\bs q^*)$ as follows:

\begin{equation}
\begin{split}
 & H(\bs z)-\eta H(\bs q^*) \\
 & = (A^{-1} + Z + (I-Z)M)^{-1} \bs z \\
 &\quad -\eta (A^{-1} + Q^* + (I-Q^*)M)^{-1} \bs q^* \\
 &\leq (A^{-1} + Z + (I-Z)M)^{-1} (I -(A^{-1} \\
 &\quad + Z + (I-Z)M) (A^{-1} + Q^* + (I-Q^*)M)^{-1}) \eta \bs q^* \\
  &= W^{-1}(I-I-((Z-Q^*)-(Z-Q^*)M)) (A^{-1} + Q^* \nonumber \\
  & \qquad + (I-Q^*)M)^{-1}) \eta \bs q^* \\
  &= - W^{-1}(Z-Q^*)(I-M) (A^{-1} + Q^* + (I-Q^*)M)^{-1} \eta \bs q^* \\
  &= - W^{-1}(Z-Q^*)F^* (A^{-1} + Q^* + (I-Q^*)M)^{-1}\eta \bs q^* \ll \bs 0,
\end{split}
\end{equation}
 where the last inequality uses the result that, $W^{-1}$ and $(A^{-1} + Q^* + (I-Q^*)M)^{-1}$ inverse of non-singular M-matrices are non-negative and non-singular (hence with no zero rows), that $F^*$ is non-negative with no zero rows, that $Z-Q^*$ has all elements positive and, that $\bs q^* \gg \bs 0$ by assumption. Consequently
\begin{equation}
    p^*_k = H_k(\bs p^*)\leq H_k(\bs z)< \eta H_k(\bs q^*) = \eta q^*_k,
\end{equation}
Since, $\eta q^*_k = p^*_k$ by definition, if $\eta > 1$, we have from above $p^*_k<p^*_k$, a contradiction. Hence, $\eta \leq 1$ which implies $\bs p^* \leq \bs q^*$. By switching the roles of $\bs p^*$ and $\bs q^*$ and repeating the above argument we can show $\bs q^* \leq \bs p^*$. Thus $\bs p^* = \bs q^*$ and hence there is a unique strictly positive fixed point.
\end{proof}


\vs{
\subsection{An illustrative example of the reduction in Section~\ref{sec: analysis-non-strong-connected}}\label{app:low-dim}
We further illustrate the system of equations in Proposition~\ref{prop:model} and Corollary~\ref{corr:reduced-model} using a simple two-layer graph shown in Fig.~\ref{fig:illus-graph}.
The reader may also refer to the technical report~\cite{VA-VS:19c} for a single-layer description of these dynamics. 

\begin{figure}[ht!]
    \centering
    \includegraphics[width=0.6\linewidth]{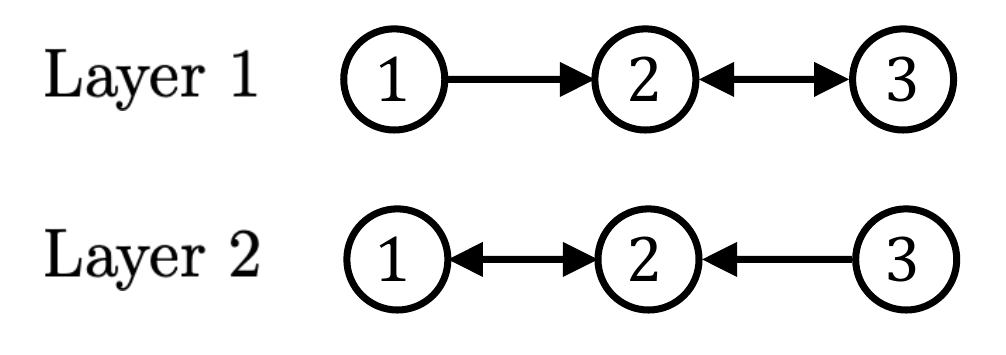}
    \caption{An example two-layer population dispersal graph. Nodes $2$ and $3$ form a strongly connected sink component in layer 1, while nodes $1$ and $2$ form a strongly connected sink component in layer 2.}
    \label{fig:illus-graph}
\end{figure}

For the two-layer graph in Fig.~\ref{fig:illus-graph}, the generator matrices associated with the population dispersal Markov chain are
\[
Q^1 = \left[\begin{smallmatrix}
-q_{12}^1 & q^1_{12} & 0 \\
0 & -q_{23}^1 & q_{23}^1 \\ 0 & q_{32}^1 & -q_{32}^1 
\end{smallmatrix}\right], \; \text{and} \; Q^2 = \left[\begin{smallmatrix}
-q_{12}^2 & q^2_{12} & 0 \\
q^2_{21} & -q_{21}^2 & 0 \\ 0 & q_{32}^2 & -q_{32}^2 
\end{smallmatrix}\right].
\]
The Laplacian matrices associated with each layer are
\begin{align*}
L^1(\bs x) &= \left[\begin{smallmatrix}
0 & 0 & 0 \\
- \frac{q_{12}^1 x_1^1}{x_2^1} &  \frac{q_{12}^1 x_1^1+ q_{32}^1 x_3^1}{x_2^1}  &  - \frac{q_{32}^1 x_3^1}{x_2^1} \\
 0 & - \frac{q_{23}^1 x_2^1}{x_3^1} &  \frac{q_{23}^1 x_2^1}{x_3^1}
\end{smallmatrix}\right], \; \text{and} \\
L^2(\bs x) &= \left[\begin{smallmatrix}
\frac{q^2_{21} x_2^2}{x_1^2} & -  \frac{q^2_{21} x_2^2}{x_1^2} & 0 \\
- \frac{q^2_{12} x_2^1}{x_2^2} &  \frac{q^2_{12} x_2^1+q^2_{32}  x_3^1}{x_2^2}  &  \frac{-q^2_{32} x_3^1}{x_2^2} \\
0 & 0 & 0
\end{smallmatrix}\right].
\end{align*}
Additionally, 
\[
F^j(\bs x) = \left[\begin{smallmatrix}
\frac{x_1^j}{x_1^1+x_1^2} & & \\
& \frac{x_2^j}{x_2^1+x_2^2} & \\
& & \frac{x_3^j}{x_3^1+x_3^2}
\end{smallmatrix}\right], \; j \in \{1,2\}.
\]
For brevity, we focus on dynamics in layer 1. The dynamics~\eqref{eq_p_alpha} can be specialized to this example to get
\begin{align}\label{eq:example-layer1}
  \dot{\bs p}^1 = (-D^1 - L^1(\bs x^1))\bs p^1 + (I - P^1) B^1 \big( F^1(\bs x) \bs p^1 + F^2(\bs x) \bs p^2 \big).  
\end{align}

Note that nodes $2$ and $3$ form a strongly connected sink component in layer 1, while nodes $1$ and $2$ form a strongly connected sink component in layer 2. Thus, $\bar v_1^1 =2, \bar v_2^1 =3, \bar v_1^2=1,$ and $\bar v_2^2 =2$. Additionally, $\hat v_1^1 = 1$ and  $\hat v_1^2 = 3$. For the reduced model, we define
\begin{align*}
    \bar{\bs p}^1 = \left[\begin{smallmatrix}
p_2^1 \\ p^1_3 \end{smallmatrix}\right], \; 
  \bar{\bs p}^2 =  \left[\begin{smallmatrix}
p_1^2 \\ p^2_2 \end{smallmatrix}\right], \; \bar{\bs x}^1 = \left[\begin{smallmatrix}
x_2^1 \\ x^1_3 \end{smallmatrix}\right], \; 
  \bar{\bs x}^2 =  \left[\begin{smallmatrix}
x_1^2 \\ x^2_2 \end{smallmatrix}\right],
\end{align*}
$\hat p^1 = p_1^1$, $\hat p^2 = p_3^2$, $\hat x^1 = x_1^1$, and $\hat x^2 = x_3^2$. We now write dynamics~\eqref{eq:example-layer1} restricted to nodes $2$ and $3$. 
\begin{multline}\label{eq:example-reduced}
    \dot{\bar{\bs p}}^1 = - \bar{D}^1 \bar{\bs p}^1 - \underset{=\bar L^1(\bs x)}{\underbrace{\left[\begin{smallmatrix}
 \frac{q_{12}^1 x_1^1+ q_{32}^1 x_3^1}{x_2^1}  &  - \frac{q_{32}^1 x_3^1}{x_2^1} \\
- \frac{q_{23}^1 x_2^1}{x_3^1} &  \frac{q_{23}^1 x_2^1}{x_3^1}
\end{smallmatrix}\right]} }\bar{\bs p}^1 - \underset{= \hat L^1(\bs x)}{\underbrace{\left[\begin{smallmatrix}
- \frac{q_{12}^1 x_1^1}{x_2^1}  \\
 0 
\end{smallmatrix}\right]}} \hat p^1 \\
+ (I - \bar P^1) \bar B^1 \Bigg(\underset{= \bar F^{11}(\bs x)}{\underbrace{\left[\begin{smallmatrix}
\frac{x_2^1}{x_2^1+x_2^2} & 0\\
0 & \frac{x_3^1}{x_3^1+x_3^2}
\end{smallmatrix}\right]}} \bar{\bs p}^1 +  
\underset{= \hat F^{11}(\bs x)}{\underbrace{\left[\begin{smallmatrix}
0 \\
0
\end{smallmatrix}\right]}} \hat{ p}^1 \\
+ \underset{= \bar F^{12}(\bs x)}{\underbrace{\left[\begin{smallmatrix}
0 & \frac{x_2^2}{x_2^1+x_2^2} \\
0 & 0
\end{smallmatrix}\right]}} \bar{\bs p}^2 + 
\underset{= \hat F^{12}(\bs x)}{\underbrace{\left[\begin{smallmatrix}
0 \\
\frac{x_3^2}{x_3^1+x_3^2}
\end{smallmatrix}\right]}} \hat{ p}^2
\Bigg).
\end{multline}
It can be verified that dynamics~\eqref{eq:example-reduced} is consistent with Corollary~\ref{corr:reduced-model} and the matrices defined in~\eqref{eq:example-reduced} are consistent with the definitions in Section~\ref{sec: analysis-non-strong-connected}.}
\end{document}